\documentclass[12pt]{amsart}
\usepackage{amssymb}
\usepackage{amsfonts}

\usepackage[french,english]{babel}
\newtheorem{theorem}{Theorem}[section]
\newtheorem{definition}{Definition}[section]
\newtheorem{proposition}{Proposition}[section]
\newtheorem{lemma}{Lemma}[section]
\newtheorem{corollary}{Corollary}[section]

\newtheorem{remark}{Remark}[section]

\newcommand{\R}{\ensuremath{\mathbb{R}}}

\newcommand{\car}{1\hspace{-1.4 mm}1}
\def\xb{{\bar x}}

\numberwithin{equation}{section}

\title[Asymptotic behavior of HJ equation]
{Large time behavior of solutions of  viscous Hamilton-Jacobi
Equations with superquadratic Hamiltonian}\thanks{This work was partially supported
by the ANR ``Hamilton-Jacobi et théorie KAM faible'' (ANR-07-BLAN-3-187245), the AUF (Agence Universitaire de la Francophonie) scholarship
program and the SARIMA (Soutien aux Activit\'es de
Recherche d'Informatique et de Math\'ematiques en Afrique) project.}
\author[Thierry  Tabet Tchamba]{ Thierry Tabet Tchamba }

\begin{document}

\date{}

\maketitle

\bigskip
\begin{small}
\noindent{\bf Abstract }
We study the long-time behavior of the
unique viscosity solution $u$ of the viscous Hamilton-Jacobi
Equation $u_t-\Delta u+|Du|^m=f\hbox{ in }\Omega\times (0,+\infty)$
with inhomogeneous Dirichlet boundary conditions, where $\Omega$
is a bounded domain of $\mathbb{R}^N$. We mainly focus on the superquadratic case
($m>2$) and consider the Dirichlet conditions in the generalized
viscosity sense. Under rather natural assumptions on $f,$ the initial and boundary data,
we connect the problem studied to its associated stationary 
generalized Dirichlet problem on one hand and to a stationary problem with a state 
constraint boundary condition on the other hand.
\end{small}
\vskip1em

\bigskip

\section{ Introduction }

The motivation of this work is the study of the large time behavior of the unique solution of a nonlinear second order parabolic 
equation of the following type 
\begin{equation}\label{1}
u_{t}-\Delta u + |Du|^{m}=f(x) \quad\hbox{ in } \Omega \times (0,+\infty),
\end{equation} where $\Omega$ is a bounded domain of $\mathbb{R}^{N}$ with a $C^{2}$-boundary, $m>2$ and $f\in C(\overline{\Omega}).$ The solution $u$ is a real-valued function defined on $\overline{\Omega} \times [0, +\infty)$ and $u_t,$ $Du,$ $\Delta u$ denote respectively the partial derivative with respect to $t,$ the gradient
with respect to the space variable and the Laplacian of $u.$ We complement (\ref{1}) with initial and boundary conditions,
namely
\begin{equation}\label{1'}
u(x,0)=u_{0}(x) \quad\hbox{ in }\,\  \overline{\Omega},
\end{equation}
\begin{equation}\label{1''}
u(x,t)=g(x)  \quad\hbox{ on }\,\  \partial\Omega \times
[0, +\infty),
\end{equation}
where $u_{0}:\Omega \rightarrow \mathbb{R}$ and
$g:\partial\Omega\rightarrow \mathbb{R}$ are bounded and
continuous functions satisfying the compatibility condition
\begin{equation}\label{2}
u_{0}(x)=g(x) \quad\hbox{ for all } x\in \partial \Omega.
\end{equation}

In the existing litterature, as far equation (\ref{1}) is concerned, some works have been devoted to its
study of in the whole space $\mathbb{R}^N,$ addressing the
questions of the existence, uniqueness and properties of either
classical solutions (see for example Amour \& Ben-Artzi \cite{AmBe}, Ben-Artzi \cite{Ben1,Ben2},
Gilding, Gueda \& Kerner \cite{GilGueKer} and the references therein) or solutions in the sense of distributions
functions (see for example Ben-Artzi, Souplet and Weissler \cite{BenSouWei}). 

Some other works, like for instance, Fila \& Lieberman \cite{FiLie}
and Souplet \cite{Souplet} deal with (\ref{1}) in an open bounded subset of $\mathbb{R}^N$ by studying the solvability of the
Cauchy-Dirichlet problem (\ref{1})-(\ref{1'})-(\ref{1''}). They prove that, under suitable assumptions on $u_0$ and $g,$ there exists a solution on some interval $[0,T^*),$ with the property that its gradient blows up on the boundary $\partial\Omega$ while the
solution remains bounded. This singularity thus yields a
difficulty when one wants to extend the solution past $T^*.$

Recently, Barles \& Da Lio \cite{BaDa1} proved that, actually when
$1<m\leq2,$ the Cauchy-Dirichlet problem can be solved in the
classical sense but can not longer be solved, that way, for any $g$ when
$m>2.$ More precisely, when $g$ is large, there could be loss
of boundary condition when $m>2,$ due to the presence of the
superquadratic growth in $|Du|.$
They consider (\ref{1})-(\ref{1'})-(\ref{1''}) with $f=0$ and prove
that there exists a unique continuous, global in time solution of (\ref{1})-(\ref{1'}) with
the Dirichlet boundary condition in the relaxed viscosity formulation  (see Definition \ref{visc} and \cite[Theorem 3.1]{BaDa1}). Moreover, they
provided an explicit expression of the solution of
(\ref{1})-(\ref{1'})-(\ref{1''}) in terms of a value function of
some stochastic control exit time problem obtained by considering the state $(X_t)_t$ of
a system driven by the stochastic differential equation
\begin{equation}\label{SDE}
dX_t=a_tdt+dB_t \ \hbox{ for } t>0, \ \ X_0=x \in \Omega
\end{equation} where $(B_t)_t$ is a
$N$-dimensional Brownian motion and  $(a_t)_t,$ the control, is some progressively
measurable process with respect to the filtration associated to $(B_t)_t$ which takes values in $\mathbb{R}^N.$ They proved that, for all
$(x,t)\in \overline{\Omega}\times [0,+\infty),$ the value function
\begin{eqnarray}\label{valfunct}
u(x,t):=\underset{(a_s)_s} \inf \ \mathbb{E}_{x}\biggl\{\int
_{0}^{\tau_x}\biggl[f(X_s)+(m-1)m^{-\frac{m}{m-1}}|a_s|^{\frac{m}{m-1}}\biggr]
ds\\ +\car_{\tau_x\leq t}g(X_{\tau_x})+\car_{\tau_x>
t}u_0 (X_t )\biggr\}\nonumber
\end{eqnarray} is continuous in $\Omega\times [0,T]$
for all $T>0$ and its continuous extension on $\overline{\Omega}
\times [0,T]$ is the unique viscosity solution of
(\ref{1})-(\ref{1'})-(\ref{1''}) (see \cite[Theorem 3.2]{BaDa1})
where $\mathbb{E}_{x}$ represents the conditional expectation with respect to the event
$\{X_0=x\}$ and $\tau_x$ is the first time when the trajectory $(X_t)_t,$ starting at $x,$ hits the
boundary $\partial\Omega.$ 

Our paper mainly complements the investigation of \cite{BaDa1} by
analyzing the large time behavior of the global solution of
(\ref{1})-(\ref{1'})-(\ref{1''}). Several papers have
studied this question either for the Cauchy problem for (\ref{1}) in the
whole space $\mathbb{R}^N$ (see for example \cite{BaPa1}, \cite{BeKaLa} and \cite{Gilding1}) or for the
Neumann problem for (\ref{1}) (see \cite{BaDa2}, \cite{BeDa1} and \cite{DaLio2}) but,
to the best of our knowledge, the first and only work to deal the
Cauchy-Dirichlet problem (\ref{1})-(\ref{1'})-(\ref{1''}) is the
one of Benachour, Dabuleanu-Hapca \& Laurencot \cite{BeDaLa}. They studied the long time
behavior of the solution of (\ref{1}) with $f\equiv 0$ associated with the homogeneous classical Dirichlet
boundary condition ($g\equiv 0$). Their main results can be roughly summarized in
the following way: (first) for $m>1,$ the global classical
solution (in the sense of \cite[Definition 1.1]{BeDaLa}) decays to
zero in the $W^{1,\infty}$-norm with the same rate as in the
linear case; (next) for $m=1,$ and exponential decay to
zero also take place, but the rate of convergence differs from
that of the linear case; (finally) when $m\in (0,1),$
the gradient term  plays a role in the large time
dynamics and a finite time extinction occurs for the nonnegative
solution.

In the general case, one may think at first glance that the solution $u$ of
(\ref{1})-(\ref{1'})-(\ref{1''}) converges uniformly to the unique
solution of the stationary equation
\begin{equation}\label{6a}
-\Delta w(x) + |Dw(x)|^{m}= f(x) \quad\hbox{ in }\Omega
\end{equation} associated with the Dirichlet boundary condition
\begin{equation}\label{6b}
w(x) = g(x)\quad\hbox{ in }\Omega,
\end{equation} and this phenomenon is consistent with the results obtained in \cite{BeDaLa} for which
$f\equiv0$ and $g\equiv 0$ and therefore the unique solution of (\ref{6a})-(\ref{6b}) is $w\equiv 0.$ But, for any $m>1,$ the existence of a
viscosity solution for (\ref{6a})-(\ref{6b}) is no longer guaranted in general when $f\neq0$ (see Alaa \& Pierre \cite{AlPi} or Grenon,  Murat \& Porretta \cite{GreMuPo} or Souplet \& Zhang \cite{SouZha} for details). For the reader convenience, we provide an example showing that (\ref{6a}) do not have any solution for any bounded $f,$ even smooth enough. Moreover we also note that the uniform $L^{\infty}$-bounds of
$u$ is not always guaranted when $f$ is bounded even for simple ordinary ($y'(t)=-1$) or partial $( u_{t}-u_{xx}+|u_{x}|=-1)$ equations. It is then be hopeless for fully
nonlinear PDE as we can note from (\ref{valfunct}) where $u(x,t)$ could go to $-\infty$ for some $f<0.$

Taking this into account, we find that the study of the long time behavior of the solution $u$ of (\ref{1})-(\ref{1'})-(\ref{1''})
first lead us to the study of a stationary ergodic problem. More precisely, like in the work of lasry \& Lions \cite{LaLi}, we are interested in findind an appropriate constant $c$ such that
the function $u(\cdot,t)+ct$ remains bounded and $c$ is the unique constant for which the stationnary equation
with state constraint boundary condition
\begin{equation}\label{112a}
-\Delta w(x) +|Dw(x)|^m=f(x)+c \quad\hbox{ in }\,\ \Omega
\end{equation}
\begin{equation}\label{112b}
-\Delta w(x) + |Dw(x)|^{m}\geq f(x)+c \quad\hbox { on }\,\ \partial \Omega.
\end{equation}
has a continuous and bounded viscosity solution $u_\infty.$

The final goal of the present paper is to establish, when $m>2,$ that
\begin{itemize}
\item[(i)] either $c<0$ and $u(\cdot, t)$ converges to
the unique solution of the generalized Dirichlet problem (\ref{6a})-(\ref{6b}) as $t \to +\infty,$
\item[(ii)] or $c=0$ and $u(\cdot, t)$ converges to a solution of the generalized Dirichlet problem (\ref{6a})-(\ref{6b}) as $t \to +\infty,$
\item[(iii)]  or $c>0$ and, independently of $g,$ the function $u(\cdot,t)+ct$ converges uniformly on $\Omega$
to a solution of the ergodic problem  with the state constraint boundary condition (\ref{112a})-(\ref{112b}).
\end{itemize}

To briefly explain such a behavior in the case where $c>0,$ we
notice that the function $u(x,t)+ct$ is the unique viscosity
solution of (\ref{1})-(\ref{1'})-(\ref{1''}) in which $f(x)$ and $g(x)$ are replaced  by $f(x)+c$ and $g(x)+ct$ respectively.
When $m>2,$ loss of boundary condition really happens in this case since
we prove that $u(x,t)+ct$ remains bounded on $\overline{\Omega}.$ This loss of boundary condition roughly introduces the
state constraints suggesting that in (\ref{valfunct}), 
the first time when the trajectory $(X_t)_t$ hits the boundary is $\tau_x=+\infty.$ This means 
that there exists a control $(a_t)_t$ keeping the process
$(X_t)_t$ inside $\Omega,$ for all $t\geq0$ with probability one.

The problem (\ref{112a})-(\ref{112b}) therefore naturally introduces the state constraint problems
coming from stochastic control problems, we refer the reader to  \cite{LaLi} where
this topic has largely been studied. We recall that state constraint problems were first studied by
Soner \cite{Soner1} in the deterministic case (see also
Fleming and Soner \cite{FleSon} and Capuzzo-Dolcetta and Lions
\cite{CapLi}) whereas Katsoulakis \cite{Kat} and Lasry and Lions
\cite{LaLi} studied it from the stochastic point of view.

This paper is organized as follow: in Section \ref{section1}, we first briefly
recall in which sense the viscosity solution of a general initial boundary-value problem has to be understood. Next,  
we recall some results on the strong comparison (uniqueness) and global in time existence 
for a parabolic problem and its associated stationary problem. We continue by introducing 
some results on the Strong maximum Principle for some linear parabolic and elliptic equation 
and we finally provide a simple
example showing that the one dimensional problem (\ref{6a})-(\ref{6b}) cannot be solved for any function $f.$ In Section \ref{section3} we prove the existence of the pair $(u_\infty,c)$ solution 
of (\ref{112a})-(\ref{112b}) with some related properties.
In Section \ref{section4}, we state and prove the final convergence result.

\bigskip

\section{Preliminaries.}\label{section1}
Given an open bounded and connected subset $\mathcal{O}$ of $\mathbb{R}^N$ with a $C^2$-boundary, let $F\in C(\overline{\mathcal{O}}\times[0,T]\times\mathbb{R}\times\mathbb{R}^N\times\mathcal{S}_N),$ $\phi\in C(\partial\mathcal{O}\times[0,T])$ and $\psi\in C(\overline{\mathcal{O}})$ be such that 
\begin{equation}\label{1'*}
\phi(x,0)=\phi(x) \quad\hbox{ for all }  x \in\partial\mathcal{O}.
\end{equation} By viscosity subsolution, supersolution and solution of the generalized initial boundary value problem
\begin{equation}\label{parabolic}
(BVP)\quad\quad\left\{
\begin{array}{rl}
u_t+F(x,t,u,Du,D^2u)&=0 \quad\quad\quad\hbox{ in }\mathcal{O}\times(0,T]\nonumber\\
u(x,t)&= \phi(x,t) \quad\hbox{ on } \partial\mathcal{O}\times(0,T]\nonumber\\
u(x,0)&= \psi(x) \ \ \quad\hbox{ on } \overline{\mathcal{O}},
\end{array}
\right.
\end{equation}
we mean the following 
\begin{definition}\label{visc}
\item[\ (i)]
An upper semicontinuous (usc in short) function $u$ in $\overline{\mathcal{O}} \times [0, T )$ is a
viscosity subsolution of (BVP) if and only if,
for all $\varphi \in C^{2}(\mathcal{O} \times [0, T]),$ at any local
maximum point $(x_{0}, t_{0})$ of $u-\varphi$ in
$\overline{\mathcal{O}} \times (0, T],$ the following holds
\begin{equation}\label{72}
\left\{
\begin{array}{rl}
\dfrac{\partial \varphi}{\partial t}(x_{0}, t_{0})+F(x_0,t_0,u,D\varphi,D^2\varphi) &\leq 0\quad\hbox{ if } (x_{0}, t_{0}) \in \mathcal{O} \times (0, T), \nonumber\\
\min\biggl\{\dfrac{\partial \varphi}{\partial t}+F(x_{0}, t_{0},u,D\varphi,D^2\varphi),
(u-\phi)(x_{0}, t_{0})\biggr\}&\leq 0 \quad\hbox{ if } (x_{0},t_{0}) \in \partial\mathcal{O} \times (0, T),\nonumber\\
u(x_{0},0)&\leq\psi(x_{0})\hbox{ if } (x_{0},t_{0}) \in
\overline{\mathcal{O}} \times \{0\}.\nonumber
\end{array}
\right.
\end{equation}
\item[\ (ii)]  A lower semicontinuous (lsc in short) function $u$ in $\overline{\mathcal{O}} \times
[0, T ]$ is a viscosity supersolution of (BVP) if and only if,
for all $\varphi \in C^{2}(\mathcal{O} \times [0, T]),$ at any local
minimum point $(x_{0}, t_{0})$ of $u-\varphi$ in
$\overline{\mathcal{O}} \times (0, T],$ the following holds
\begin{equation}\label{72*}
\left\{
\begin{array}{rl}
\dfrac{\partial \varphi}{\partial t}(x_{0}, t_{0})+F(x_0,t_0,u,D\varphi,D^2\varphi) &\geq 0\quad\hbox{ if } (x_{0}, t_{0}) \in \mathcal{O}\times (0, T), \nonumber\\
\max\biggl\{\dfrac{\partial \varphi}{\partial t}+F(x_{0},t_{0},u,D\varphi,D^2\varphi),
(u-\phi)(x_{0}, t_{0})\biggr\}&\geq 0 \quad\hbox{ if } (x_{0},t_{0}) \in \partial\mathcal{O} \times (0, T),\nonumber\\
u(x_{0},0)&\geq\psi(x_{0})\hbox{ if } (x_{0},t_{0}) \in
\overline{\mathcal{O}} \times \{0\}.\nonumber\end{array}
\right.
\end{equation}
\item[ \ (iii)] $u$ is a (continuous) viscosity
solution of (BVP) if $u$ is both a sub
and a supersolution.
\end{definition}
We use the term ``generalized" because we want to stress  on the fact that the viscosity solution theory
leads to a new formulation of the boundary conditions for degenerate elliptic equations,
that is why the standard Dirichlet boundary condition ``$u=\phi$ on  $\partial\mathcal{O} \times (0, T)$" has to be relaxed in the sense of Definition \ref{visc}. We should also write relaxed conditions on $\partial\mathcal{O} \times \{0\}$ and $\mathcal{O} \times \{0\},$ but for such parabolic
equations, by using the compatibility condition (\ref{1'*}), we find that the classical condition always holds at $t=0$ as mentioned in \cite[Lemma 4.1]{DaLio1} (see also \cite[Theorem 4.7]{Barles1} for first order equations). 

The definition of a viscosity solution for stationary elliptic operator is similar, we refer the reader to the books of Barles \cite{Barles1}, Bardi \& Capuzzo-Dolcetta \cite{BarCap}, Koike \cite{Koike1} and the user's guide of Crandall, Ishii \& Lions \cite{CranILi} for more details on the viscosity solution theory.

In particular, when $\mathcal{O}=\Omega,$ $F=-\Delta u+|Du|^m-f,$ $\phi=g$ and $\psi=u_0,$ we recall from \cite{BaDa1} that when $0<m\leq 2,$ there is no loss of boundary condition for sub and supersolutions: for any
subsolution $u$ (resp. supersolution $v$), we always have $u\leq \phi$ (resp. $v\geq \phi$) on
$\partial\Omega\times (0,T].$ For $m>2,$ we still have $u\leq \phi$ on
$\partial\Omega\times (0,T],$ but losses of boundary condition could happen for supersolution.
These results can be found in \cite[Propositions 3.1 and 3.2]{BaDa1} for the case $f=0$ 
and the proofs of  \cite{BaDa1} can easily extend to the general case.

Hereafter, we denote by $E(\mathcal{O},f,g,u_0,\lambda)$ the generalized initial boundary value problem $(BVP)$ when
$F=-\Delta u+|Du|^m+\lambda u-f$ with $\lambda\geq 0,$ $\phi=g$ and $\psi=u_0.$ We also notice that $E(\mathcal{O},f,g,u_0,0)$ and (\ref{1})-(\ref{1'})-(\ref{1''}) represent exactly the same problem. Here ``$E$" stands for evolution. 
\subsection{Comparison and existence results for a nonlinear parabolic problem.}\label{existence}\text{}

We start with
\begin{theorem}[Strong comparison result]\label{13} \text{ }\\ Given $T>0,$ for all $\ m>0,$ assume that
$\,\, f, u_{0}\in C(\overline{\mathcal{O}}), \phi \in C(\partial\mathcal{O})
,$ let $u$ be a bounded usc viscosity subsolution and $v$ be a bounded lsc viscosity supersolution of
$E(\mathcal{O},f,g,u_0,\lambda).$ Then $u \leq v \text{ in } \mathcal{O}\times [0, T].$ Moreover, if we define $\tilde{u}$ on $\overline{\mathcal{O}}\times [0,T)$ as follow:
\begin{equation}\label{72'}
\tilde{u}(z,t):=\left\{
\begin{array}{rl}
&\underset{ \underset {(y,s) \rightarrow (z,t)}  {(y,s)\in
\mathcal{O}  \times (0, T]} } {\limsup} u(y,s) \quad\hbox{ for all } (z,t)
\in \partial \mathcal{O}\times  [0,T),\\
&u(z,t) \quad\quad\quad\quad\hbox{ for all }(z,t) \in \mathcal{O}\times  [0,T),
\end{array}
\right.
\end{equation} then $\tilde{u}$ still a bounded subsolution of $E(\mathcal{O},f,g,u_0,\lambda)$ and  $\tilde{u} \leq v \text{ in } \overline{\mathcal{O}}\times [0, T].$
\end{theorem}
\begin{remark}\rm
In general, getting a comparison result up to the boundary with the only constraint $u(\cdot,t)\leq g \hbox{ on } \partial\mathcal{O}\times [0,T]$  is hopeless. Indeed, as we will see later on,  loss of boundary condition can sometimes happen. Let, for example, $w$ be  a usc subsolution of $E(\mathcal{O},f,g,u_0,\lambda)$ such that $w(x_0,t)< g(x_0) \hbox{ for some } x_0\in\partial\mathcal{O},$ we can defining a new function $\tilde{w}(\cdot,t)$ to be equal to $w(\cdot,t)$ in $\overline{\mathcal{O}}$ except at $x_0,$ where we set $\tilde{w}(x_0,t)=g(x_0).$ Since $\tilde{w}$ remain a viscosity subsolution of $E(\mathcal{O},f,g,u_0,\lambda),$ an application of the first part of Theorem \ref{13} yields $\tilde{w}\leq w \hbox{ in }\mathcal{O}\times [0,T]$ whereas $\tilde{w}$ is not less that $w$ on $\partial\mathcal{O}$ since $w(x_0,t)<\tilde{w}(x_0,t).$ Redefining $u$ on the boundary by (\ref{72'}) enables us to turn around that difficulty and extend the comparison up to the boundary.
\end{remark}
The proof of Theorem \ref{13} is performed exactly as the one of \cite[Theorem 3.1]{BaDa1} though the presence of $f$ and $\lambda$-term. The continuity of $f$ on $\overline{\mathcal{O}}$ being sufficient to apply the same arguments with no significant changes.

We continue with
\begin{corollary}\label{13*} Under the assumptions of Theorem \ref{13}, if $\,\,u_1$ is a bounded usc subsolution of $E(\mathcal{O},f_1,g_1,u_0,\lambda)$ such that
\begin{equation}\label{ss-bord}
 u_1 (x,t)=\underset{\underset {(y,s) \rightarrow (x,t)}  {(y,s)\in
\mathcal{O}  \times (0, T]} } {\limsup} u_1 (y,s) \quad\hbox{ for all } (x,t)
\in \partial \mathcal{O}\times  [0,T),
\end{equation}
and if $u_2$ is a bounded lsc supersolution of
$E(\mathcal{O},f_2,g_2,v_0,\lambda),$ then we have
\begin{equation}\label{31} \|(u_1-u_2)^{+}\|_\infty \leq t
\|(f_1-f_2)^{+}\|_{\infty}  + e^{-\lambda t}\|(u_{0}-v_{0})^{+}\|_{\infty}
+ \|(g_1-g_2)^{+}\|_{\infty}
\end{equation}
for all $0\leq t\leq T$ where  $a\vee b := \max (a,b),$ $r^{+}:=r \vee 0$ and $\|\cdot\|_{\infty}$ is the classical sup-norm either on
$\overline{\mathcal{O}}\times [0,T]$ or $\overline{\mathcal{O}}$ or $\partial\mathcal{O}.$
\end{corollary}
\begin{proof}[\bf Proof of Corollary \ref{13*}.]\text{} Put
$M_1:=\|(u_{0}-v_{0})^+\|_{\infty},$ $M_2:=\|(g_1-g_2)^{+}\|_{\infty},$ $M_3:=\|(f_1-f_2)^{+}\|_{\infty}$ and $\Lambda (t):=tM_3 +e^{-\lambda t}M_1+M_2.$ We claim that $\tilde{u}:=u_1-\Lambda(t)$ is a subsolution of $E(\mathcal{O},f_2,g_2,v_0,\lambda).$ Indeed, we first consider a point $(x,t)\in \mathcal{O}\times
(0,T],$ and formally compute
\begin{eqnarray}
\tilde{u}_{t}-\Delta \tilde{u}+|D\tilde{u}|^{m}+\lambda \tilde{u}&=&(u_1)_{t}-\Delta u_1 + |Du_1|^{m}+\lambda u_1-\Lambda'(t)-\lambda\Lambda(t)\nonumber\\
&\leq&f_1(x)-M_3 \leq f_1(x)-(f_1-f_2)(x)=f_2(x)\nonumber.
\end{eqnarray} Next, we obtain $\tilde{u}(\cdot,0)\leq v_{0}(\cdot)$ in $\mathcal{O}$ by computing
\begin{eqnarray}\tilde{u}(x,0)\leq u_1(x,0)-M_1\leq u_{0}(x)-(u_{0}-v_{0})(x)=v_{0}(x)\nonumber.
\end{eqnarray} Finally, if $(x,t)\in \partial\mathcal{O}\times [0,T],$ then
\begin{eqnarray}
\tilde{u}(x,t)\leq g_1(x)-M_2\leq g_1(x)-(g_1-g_2)(x)=g_2(x).\nonumber
\end{eqnarray} Thus, $\tilde{u}$ and $v$ are
respectively sub- and supersolution of $E(\mathcal{O},f_2,g_2,v_0,\lambda).$
We apply Theorem \ref{13} to obtain $\tilde{u} \leq u_2$ in
$\overline{\mathcal{O}} \times [0,T]$ thus leading to (\ref{31}).
\end{proof}
We end with the
\begin{theorem}[Existence and Uniqueness for $E(\mathcal{O},f,g,u_0,\lambda)$]\label{13''} \text{ }\\
For any $f\in C(\overline{\mathcal{O}}),$ $u_0\in C(\overline{\mathcal{O}})$ and $g\in C(\partial\mathcal{O})$ satisfying (\ref{2}),
there exists a unique, global in time, continuous viscosity solution of the generalized initial boundary-value problem $E(\mathcal{O},f,g,u_0,\lambda).$
\end{theorem}
The proof of this result consists in coupling the comparison result with
the Perron's method on the time interval $[0,T].$ We refer to Da Lio \cite{DaLio1} for a complete proof of
this result.

\begin{remark}\rm It is worth mentioning that Theorems \ref{13} and \ref{13''} still hold even when $f$ and $g$ depend on the $x$ and $t$  variables. It would be enough, in that case,  to assume that $f\in C(\overline{\mathcal{O}} \times [0, T]$ and $g\in C(\partial\mathcal{O} \times [0, T])$ for any $T>0.$
\end{remark}

\subsection{Comparison Principle for a stationary problem}\text{}

Here, we denote by $S(\mathcal{O},h,k,\lambda),$ with ``$S$" standing for stationary, the following generalized Dirichlet problem
\begin{equation} \label{stat}
\left\{
\begin{array}{rcl}
-\Delta \phi + |D \phi|^m +\lambda \phi& = &h  \quad\hbox{ in }\mathcal{O} \nonumber\\
\phi & = &k  \quad\hbox{ on  }\partial\mathcal{O}\nonumber
\end{array}
\right.
\end{equation} where $h\in C(\overline{\mathcal{O}}),$ $k\in C(\partial\mathcal{O})$ and $\lambda\geq0.$
\begin{theorem}\label{scr1}
Let $h\in C(\overline{\mathcal{O}})$ and $k\in C(\partial\mathcal{O}).$ Assume that $m>0$ and $\lambda>0$ $($resp. $m>1$ and $S(\mathcal{O},h,k,0)$ has a strict subsolution$).$  Let $u\in USC(\overline{\mathcal{O}})$ and $v\in LSC(\overline{\mathcal{O}})$ be respectively bounded viscosity sub- and supersolution of $S(\mathcal{O},h,k,\lambda)$ $(\hbox{resp. } S(\mathcal{O},h,k,0))$ then $u\leq v$ on $\mathcal{O}.$ Moreover, if we define $\tilde{u}$ on $\overline{\mathcal{O}}$ as follows:
\begin{equation}
\tilde{u}(z):=\left\{
\begin{array}{rl}
&\underset{ y \rightarrow z,\ y\in
\mathcal{O} }{\limsup} \ u(y) \quad\hbox{ for all } z
\in \partial \mathcal{O},\\
&u(z) \quad\quad\quad\quad\hbox{ for all }z\in \mathcal{O},
\end{array}
\right.
\end{equation} then $\tilde{u}$ still a bounded subsolution of $S(\mathcal{O},h,k,\lambda)$ $(\hbox{resp. } S(\mathcal{O},h,k,0))$ and  we have $\tilde{u} \leq v \text{ in } \overline{\mathcal{O}}.$
\end{theorem}
We recall that a function is a said to be a strict subsolution of
$S(\mathcal{O},h,k,\lambda)$ if it is a subsolution of
$S(\mathcal{O},h-\eta,k-\delta,\lambda)$ for some $\delta>0$ and $\eta>0.$
\begin{proof}[\bf Proof of Theorem \ref{scr1}] We start by the case where $\lambda>0.$
\\ Since $u_t=v_t=0,$ it is obvious that $u$ is a subsolution of $E(\mathcal{O},h,u,k,\lambda)$ and that $v$ is supersolution of
and $E(\mathcal{O},h,v,k,\lambda).$ Coming back to (\ref{31}), for all $0<t<T,$ we obtain
$$\|(u-v)^{+}\|_\infty \leq  e^{-\lambda t}\|(u-v)^{+}\|_{\infty}.$$
Knowing that $\lambda, t>0,$ the last inequality holds only when $\|(u-v)^{+}\|_\infty=0$ and the comparison  $u\leq v$ on $\overline{\mathcal{O}}$ thus follows.

Now, we prove the $\lambda=0$ case.\\
Let $\bar{\phi}$ be a strict subsolution of $S(\mathcal{O},h,k,0),$ since $m>1,$ the map $p\mapsto|p|^m$ is convex and for all $0<\mu<1,$ we find that $u_\mu:=\mu u+(1-\mu)\bar{\phi}$ is a also strict subsolution of $S(\mathcal{O},h,k,0).$ Indeed, we formally have : 
\begin{eqnarray}
-\Delta u_{\mu} + |Du_{\mu}|^{m}&=& -\Delta [\mu u + (1-\mu) \bar{\phi}] + |D(\mu u + (1-\mu) \bar{\phi})|^{m}\nonumber\\
& \leq& -\mu \Delta u - (1-\mu)\Delta \bar{\phi}+ \mu |Du|^{m} + (1-\mu) |D\bar{\phi}|^{m}\nonumber\\
& <& \mu f+(1-\mu)f=f\nonumber.
\end{eqnarray}
This means that there exists two constants $\eta, \delta>0$ such that $u_\mu$ is a subsolution of $S(\mathcal{O},h-\eta,k-\delta,0).$ It is then obvious that $u_{\mu,\eta}(x,t):=u_\mu(x)+\eta t$ is a
subsolution of $E(\mathcal{O},h,k-\delta +\eta t,u_\mu,0)$ and $\tilde{v}(x,t)\equiv v(x)$ is a supersolution of $E(\mathcal{O},h,k,v0).$
From (\ref{31}) in Corollary \ref{13*}, we obtain, for all $0<t<T,$ the following estimate on $u_{\mu,\eta}-\tilde{v}$
\begin{equation}\label{47}
 \|(u_{\mu,\eta}-\tilde{v})^{+}\|_{\infty} \leq \|(u_{\mu,\eta}-\tilde{v})
_{|_{t=0}}^{+}\|_{\infty} +(\eta t- \delta  )^{+}.
\end{equation}  Since (\ref{47}) is true for all $T>0,$ it remains true
for all $0< \tau \leq T,$ that is :
\begin{equation}\label{48}
u_\mu(x)-v(x)+\eta t \leq \|(u_\mu-v)^{+}\|_{\infty} + (\eta \tau- \delta  )^{+}
\hbox{ for all } \, (x,t) \in  \overline{\mathcal{O}} \times (0, \tau).
\end{equation}
 Let $\bar{x} \in  \overline{\mathcal{O}}$ such that
$M_\mu:=u_\mu(\bar{x})-v(\bar{x})=\sup_{\overline{\mathcal{O}}}(u_\mu-v),$
(\ref{48}) implies $$M_\mu+\eta t \leq M_\mu^{+} + ( \eta \tau- \delta
)^{+}.$$  If $M_\mu>0,$ then for $\tau$ such that $ \eta \tau- \delta
\leq 0,$ we have $M_\mu+\eta t \leq M_\mu$ which leads to a contradiction. Therefore, $M_\mu\leq 0,$ and we conclude that $u\leq v \hbox{ on }
\overline{\mathcal{O}}$ by sending $\mu\to 1.$
\end{proof}
\begin{remark}\rm
Another way of proving Theorem \ref{scr1} for $\lambda>0,$ when $m>1,$ is to notice that $C_{h,k}:=-\|k\|_\infty-\|h\|_\infty/\lambda$ is a strict subsolution of $S(\mathcal{O},h,k,\lambda)$ and we continue arguing exactly as in the case where $\lambda=0.$
\end{remark}
\begin{corollary}\label{scr2} Let $h\in C(\overline{\mathcal{O}})$ and $k_1, k_2\in C(\partial\mathcal{O}).$ Assume that  $m>0$ and $\lambda>0$ $($resp. $m>1$ and $S(\mathcal{O},h,k_i,0)$ has a strict subsolution, $i=1,2).$ Let $ u_1$ be a bounded usc viscosity subsolution of $S(\mathcal{O},h,k_1,\lambda)$
$(\hbox{resp. } S(\mathcal{O},h,k_1,0))$ and
$u_2$ be supersolution of $S(\mathcal{O},h,k_2,\lambda)$ $(\hbox{resp. } S(\mathcal{O},h,k_2,0))$
then for all $x \in \overline{\mathcal{O}},$
\begin{equation}\label{compar1} \|(u_1-u_2)^+\|_\infty \leq \|(k_1-k_2)^{+}\|_{\infty}.
\end{equation}
\end{corollary}
\begin{proof}[\bf Proof of Corollary \ref{scr2}]
We just remark that $u_1- \|(k_1-k_2)^{+}\|_{\infty}$ is a subsolution of $S(\mathcal{O},h,k_2,\lambda)$ and apply Theorem \ref{scr1} to obtain $u_1\leq u_2+ \|(k_1-k_2)^{+}\|_{\infty}\hbox{ on } \overline{\mathcal{O}}.$
\end{proof}
\begin{theorem}[Existence and Uniqueness for $S(\mathcal{O},h,k,\lambda)$]\label{13'''} \text{ }\\
Let $h\in C(\overline{\mathcal{O}})$ and $k\in C(\partial\mathcal{O}).$ Assume that either $m>0$ and $\lambda>0$ or $m>1$ and $S(\mathcal{O},h,k,0)$ has a strict subsolution.  Then there exists a unique continuous viscosity solution of the generalized boundary-value problem $S(\mathcal{O},h,k,\lambda).$
\end{theorem}
\begin{proof}[\bf Proof of Theorem \ref{13'''}]
Given that the Strong comparison result holds for the stationary problem $S(\mathcal{O},h,k,\lambda)$ for $\lambda\geq0$ (see Theorem \ref{scr1}), it just remains to build appropriate sub- and supersolution of $S(\mathcal{O},h,k,\lambda)$ in order to apply the Perron's method (see \cite{Ishi}). 
On one hand, Let $x_0\in \mathbb{R}^N$ be such that $B(x_{0},K)\cap \overline{\Omega}=\emptyset$ with $K\geq(\|h\|_{\infty}+2N)^{1/m},$ We claim that $l(x)= |x-x_{0}|^{2}+ \|k\|_{\infty}+1$ is a supersolution of $S(\mathcal{O},h,k,\lambda)$ with $\lambda\geq0.$ Indeed, 
\begin{equation}
-\Delta l + |Dl|^{m}+\lambda l\geq -\Delta l + |Dl|^{m}=-2N+2^{m}|x-x_{0}|^{m}>-2N+(2K)^{m}\nonumber.
\end{equation}
To conclude that $-\Delta l + |Dl|^{m}+\lambda l\geq f,$ it's sufficient to
choose $K$ such that $(2K)^{m}\geq\|h\|_{\infty}+2N.$ By its very definition, $l>k,$ we therefore conclude that $l$ is a supersolution of $S(\mathcal{O},h,k,\lambda).$ On the other hand, if $\lambda>0$ then $-\frac{\|h\|_\infty}{\lambda}-\|k\|_\infty-1$ is a subsolution of $S(\mathcal{O},h,k,\lambda)$ whereas the assumption on $S(\mathcal{O},h,k,0),$ itself, provide a subsolution in the case where $\lambda=0.$
The existence result for $S(\mathcal{O},h,k,\lambda)$ for all $\lambda\geq 0$ therefore follows. 
\end{proof}

\subsection{The Strong Maximum Principle}\text{}

We refer the reader to Evans \cite{Evans} and Gilbarg \&Trudinger \cite{GilTru} for more about the strong maximum principle for smooth solutions of linear parabolic and elliptic equations and to Bardi \& Da Lio \cite{BarDa} and Da Lio \cite{DaLio3} for viscosity solutions of fully nonlinear degenerate elliptic and parabolic operator. The result we are concerned with is the following
\begin{lemma}\label{linear*} Let $C>0.$ Any upper semicontinuous viscosity subsolution of
\begin{equation}\label{linearization}
-\Delta\phi-C|D\phi|=0\ \hbox{ in }\ \mathcal{O}
\end{equation} $(\hbox{resp.}$ 
\begin{equation}\label{linearization1}
 \phi_t-\Delta \phi-C|D\phi| =
0\ \ \hbox{ in }\ \mathcal{O}\times (0,T]\ )
\end{equation} that attains its maximum at some $x_0\in \mathcal{O}$  $(\hbox{resp. } (x_0,t_0)\in \mathcal{O}\times (0,T])$ is constant on $\mathcal{O}$
$(\hbox{resp. } \mathcal{O}\times [0,t_0])$. In particular,
$$ \max_{\overline{\mathcal{O}}}\phi = \max_{\partial \mathcal{O}}\phi \; ,$$ $( \hbox{resp.}$
$$ \max_{\overline{\mathcal{O}}\times [0,T]}\phi = \max_{\partial_p( \mathcal{O}\times (0,T))}\phi \; ,$$
where $\partial_p( \mathcal{O}\times (0,T))$ is the parabolic boundary of $\mathcal{O}\times (0,T)$, i.e. $(\partial \mathcal{O}\times (0,T)) \cup (\overline{\mathcal{O}} \times \{0\}).)$
\end{lemma}

This result is already proved in \cite[Corollary 2.4]{BarDa} and \cite[Corollary
2.4]{DaLio3}. Actually, it is obvious to see that (\ref{linearization}) belongs to the family of nonlinear elliptic equations of the form
\begin{equation}\label{model1}
c(x)|u|^{k-1}u-\Delta u + b(x)|Du|^p=0 \quad\hbox{ in } \mathcal{O}
\end{equation} with $b:=C<0,$ $p=1$ and $c\equiv0$ whereas (\ref{linearization1}) belongs to
the family of nonlinear parabolic equations of the form
\begin{equation}\label{model2}
u_t+c(x,t)|u|^{k-1}u-a(x,t)F(Du,D^2u)=0 \quad\hbox{ in }\mathcal{O}\times [0,T]
\end{equation} where $F:=Tr(X)+C|p|$ is positively homogeneous of degree $\alpha=1,$  $c\equiv0$ and $a\equiv 1.$
We then refer to \cite{BarDa} and \cite{DaLio3} for the test of the non degeneracy conditions and the scaling properties on
(\ref{linearization}) and (\ref{linearization1}).

\subsection{When is the stationary Dirichlet problem solvable?}\text{}

As earlier noted  in this work, one could think at first glance that the needed limit problem is $S(\Omega,f,g,0)$ (or
(\ref{6a})-(\ref{6b})). Here, we briefly address the question of the existence of a solution
for (\ref{6a}) for all $f$ through an illustration by a basic example taken from \cite{BaDa1} and show that (\ref{6a}) is not solvable for any bounded continuous $f.$

To see that, let $R$ and $C$ be two positive constants, we take $f:=-C^m<0$ and study
the one-dimensional problem consisting of finding solutions of
\begin{equation}\label{example}
-\eta''+|\eta'|^m=-C^m \ \ \ \hbox{ in } \ \ \ (-R,R).
\end{equation} To solve (\ref{example}), we integrate once and after some easy change of variable and computations, we find that $\eta'$
solves the equation $$
\dfrac{1}{C^{m-1}}\int_{\frac{\eta'(0)}{C}}^{\frac{\eta'(x)}{C}}
\dfrac{ds}{|s|^m+1}=x.
$$ It then follows that
$$C^{m-1}x=\int_{\frac{\eta'(0)}{C}}^{\frac{\eta'(x)}{C}}
\dfrac{ds}{|s|^m+1}\leq\int_{-\infty}^{+\infty}
\dfrac{ds}{|s|^m+1}
<\infty$$ since $m>1.$ Therefore, letting $x\rightarrow R,$ we obtain :
\begin{equation}\label{example*}
C^{m-1}\leq \dfrac{1}{R}\int_{-\infty}^{+\infty}
\dfrac{ds}{|s|^m+1}.
\end{equation}
The inequality (\ref{example*}) says roughly that for a given
interval $[-R,R],$ the ordinary differential equation
(\ref{example}) is not solvable for large $C,$ meaning that (\ref{6a})
is not solvable when $f<0$ is such that $|f|\gg1.$ The condition on the size of $f,$ as expressed above only occurs when $f<0.$ Indeed,
if $f\geq0,$ then the constant $-\|g\|_\infty$ is a subsolution of $S(\Omega,f,g,0)$ and the existence follows from the Perron's method and the strong comparison result for $S(\Omega,f,g,0).$

To complement this question, we consider the following parameterized version of (\ref{6a}):
\begin{equation}\label{parameter}
-\Delta w + |Dw|^{m}= \epsilon f(x)\,\,\,\ \text { in }\,\ \Omega
\end{equation}
with $\epsilon>0.$ From \cite{AlPi} or \cite{GreMuPo}, we find the existence of $\epsilon^*\in (0,+\infty)$ such that (\ref{parameter}) has at least a solution for $\epsilon\in(0,\epsilon^*)$ and there is no solution when $\epsilon>\epsilon^*.$ In general, one does not know what happens if $\epsilon=\epsilon^*.$ 

This lack of existence of a solution of $S(\Omega,f,g,0)$ gives a motivation of the use of the ``ergodic problem"
(\ref{112a})-(\ref{112b}) to explain the behavior of the solution $u$ of $E(\mathcal{O},f,g,u_0,\lambda)$ for all $\lambda\geq0.$
\section{ The stationary ergodic problem }\label{section3}
Throughout this section, we aim to study the existence of the pair $(c,u_\infty)\in \mathbb{R}\times C(\overline{\Omega})$ for which $u_\infty$ is
a viscosity solution of (\ref{112a})-(\ref{112b}). To do so, we detail a new proof, by using the approach of the viscosity solutions, of the following result which already appears in \cite[Theorem VI.3]{LaLi}. To stress on the dependence of (\ref{112a})-(\ref{112b}) on $f$ and $c,$ hereafter we denote by $Erg(\Omega,f,c)$ the state constraint problem \begin{equation}\label{ergodic problem}
\left\{
\begin{array}{rl}
-\Delta w+|Dw|^m&=f+c \quad\hbox{ in } \Omega\nonumber\\
-\Delta w+|Dw|^m&\geq f+c \quad\hbox{ on } \partial\Omega.\nonumber
\end{array}
\right.
\end{equation}
\begin{theorem}\label{112}\text{}
Assume that $f\in C(\overline{\Omega})$ and $m>2.$ Then there exists  $c\in\mathbb{R}$ and a function $u_\infty\in  C^{0,\frac{m-2}{m-1}}(\overline{\Omega}) \cap W_{loc}^{1,\infty}(\Omega)$ such that $u_\infty$ is a viscosity solution of $Erg(\Omega,f,c)$. 
Moreover if $\tilde{v}$ is a viscosity solution of $Erg(\Omega,f,\tilde{c}),$
then $c=\tilde{c}$ and $u_\infty=\tilde{v}+K$ for some constant $K.$
\end{theorem}
Hereafter, we define the distance from $\partial\Omega$ to $x\in \overline{\Omega}$ by
\begin{equation}\label{distance}
d_{\partial\Omega}(x):=dist(x,\partial\Omega)=\inf \{ |x-y| \text{ with } y
\in\partial\Omega\}.
\end{equation}
For some $\delta>0,$ we denote by
\begin{equation}\label{setext}
\Omega^{\delta}:=\{x\in \Omega : d_{\partial\Omega}(x)< \delta\},
\end{equation}
\begin{equation}\label{setint}
\Omega_{\delta}:=\{x\in \Omega : d_{\partial\Omega}(x)> \delta\}.
\end{equation} 
Since $\Omega$ is $C^2,$ the function $d_{\partial\Omega}$ is also $C^2$
in a neighborhood of the boundary, say in $\Omega_{\delta}$ for all $0<\delta\leq \delta_0.$ We denote by $d$ any positive $C^2$ function
agreing with $d_{\partial\Omega}$ in $\Omega^\delta$ for all
$0<\delta\leq \delta_0$ such that $|Dd|\leq 1 \hbox{ in } \Omega_\delta$ and by $n$ a $C^1$- function defined by
\begin{equation}\label{normal}
n(x)=-Dd(x)\ \ \text{ in }\ \ \Omega^{\delta_0}.
\end{equation} If $x\in \partial\Omega,$ then $n(x)$ is just the unit outward normal vector to $\partial\Omega$ at $x.$

We are going to use the following classical result which may require to choose a smaller value of $\delta_0$.
\begin{lemma}\label{connectedness} Under the above assumptions on $\Omega$, for any $0\leq \delta\leq \delta_0$,
$\Omega_{\delta}$ is $C^1$-pathwise connected and $C^2$ domain. For $x,y \in \Omega_{\delta}$, we denote by $\mathcal{A}_{x,y}(\Omega_{\delta})$ the nonempty set
\begin{equation}\label{admissible}
\biggl\{\gamma_{x,y}:[0,1]\rightarrow\overline{\Omega_{\delta}}, \gamma_{x,y} \hbox{ is } C^1\hbox{- piecewise }  \hbox{ with } \gamma_{x,y}(0)=x, \gamma_{x,y}(1)=y\biggr\},
\end{equation}
and by $\tilde{d_\delta}$ the function defined by
\begin{equation}\label{lenght1}
\tilde{d_\delta}(x,y):=\underset{\gamma_{x,y}\in\mathcal{A}_{x,y}(\Omega_{\delta})}\inf\int_0^1|\dot{\gamma}_{x,y}(t)|dt\; .
\end{equation}
Then $\tilde{d_\delta}$ satisfies
\begin{equation}\label{lenght}
\tilde{d}(x,y)\leq C_{\Omega}
\end{equation} for some constant $C_{\Omega}>0$ and there exists ${\bar \delta}, {\bar K} >0$ such that, if $|x-y|\leq {\bar \delta},$ then
\begin{equation}\label{lenght2}
\tilde{d_\delta}(x,y)\leq {\bar K}|x-y|.
\end{equation}
\end{lemma}
The rest of this section is devoted to the proof of Theorem \ref{112}.
\begin{lemma}\label{130'}
Let $f\in C(\overline{\Omega})$ and $m>2.$ For any $R>0$ and $0<\lambda<1,$ there exists a unique
viscosity solution $u_{R, \lambda}$ of the generalized Dirichlet problem $S(\Omega,f,R,\lambda).$
Moreover, $u_{R, \lambda}$ satisfies the following estimates
\begin{equation}\label{R-Bounds}
-\frac{\|f\|_{\infty}}{\lambda}\leq u_{R, \lambda}\leq
-\frac{M}{\alpha}{d}^ {\alpha}(x) +\frac{K}{\lambda} \quad\hbox{ in } \Omega
\end{equation} with $\alpha=\frac{m-2}{m-1}$ and $M,K> 0$ are independent of $R$ and $\lambda.$
\end{lemma}
\begin{proof}[\bf Proof of Lemma \ref{130'}]
The existence of the solution $u_{R,\lambda}$ of $S(\Omega,f,R,\lambda)$ follows from a combination of the strong
comparison result for $S(\Omega,f,R,\lambda)$ (Theorem \ref{scr1}) and the classical Perron's method
(see \cite{Ishi}) with the version up to the boundary provided in
Da Lio \cite[Theorem 2.1]{DaLio1}. It is therefore sufficient to build a sub- and supersolution of $S(\Omega,f,R,\lambda).$ 

It is easy to see that the constant $-\|f\|_{\infty}/\lambda$ is a subsolution
$S(\Omega,f,R,\lambda).$ 
We claim that $ \zeta(x):=
-\frac{M}{\alpha}{d}^ {\alpha}(x)+\frac{K}{\lambda}$ is a supersolution of $S(\Omega,f,R,\lambda)$ with state constraints on $\partial\Omega.$
Indeed, in $\Omega$ we compute:
\begin{eqnarray}\label{Bounds}
-\Delta \zeta + |D\zeta|^{m}+\lambda \zeta -f&=&(\alpha
-1)Md^{\alpha -2}|Dd|^2+Md^{\alpha -1}\Delta d \\
 &&+M^{m}d^{m(\alpha -1)}|Dd|^m-\lambda\frac{M}{\alpha}d^{\alpha}+
K-f\nonumber\\
&=& \frac{M}{d^{m(1-\alpha)}}\biggl[M^{m-1}|Dd|^m+d\Delta d
\nonumber\\ &&+(\alpha-1)|Dd|^2 - \frac{\lambda}{\alpha}
d^2\biggr]+K-f.\nonumber
\end{eqnarray} In $\Omega^{\delta}$ where $|Dd|=1$ and $0\leq d\leq \delta,$
(\ref{Bounds}) reads
\begin{eqnarray}\label{bounds}
-\Delta \zeta + |D\zeta|^{m}+\lambda \zeta
-f&=&\frac{M}{d^{m(1-\alpha)}}\biggl[(\alpha-1)+d\Delta d+M^{m-1}- \frac{\lambda}{\alpha}
d^2\biggr]+K-f. \nonumber
\end{eqnarray}
It then suffices to choose the constant $M$ and $K$ as follows
\begin{equation}\label{M-bounds} M^{m-1}>(1-\alpha)+\delta\|\Delta d\|_\infty +\frac{\delta^{2}}{\alpha} \quad \text{ and } K\geq 2\|f\|_\infty
\end{equation} to get $-\Delta \zeta + |D\zeta|^{m}+\lambda \zeta
\geq f \hbox{ in } \Omega^\delta.$ Moreover, for all $x \in\Omega,$ we have
\begin{equation}
\partial_n\zeta(x):=D \zeta(x)\cdot n(x)=Md^{\alpha-1}(x)|Dd(x)|^2=Md^{-\frac{1}{m-1}}(x)|Dd(x)|^2,\nonumber
\end{equation} and one obviously obtains that
\begin{equation}
\underset{d(x)\to0}\lim\partial_n\zeta(x)=+\infty\nonumber.
\end{equation}
From an easy adaptation of \cite[Proposition 3.3]{BaDa1}, we have $J^{2,-}\zeta(x)=\emptyset,$ where
$J^{2,-}\zeta(x)$ is the second-order subjet of $\zeta$ at $x$ relative to $\overline{\Omega},$
meaning that there is no smooth test function $\varphi$ such that
$\zeta-\varphi$ achieves its local minimum point at $x\in \partial\Omega$ and thus state constraint boundary
condition is automatically satisfied.

To conclude that the function $\zeta$ is a solution of
\begin{equation}\label{130bis}
-\Delta w + |Dw|^{m}+\lambda w= f \,\,\,\,\ \text { in }\,\ \overline{\Omega}
\end{equation} it remain to prove the supersolution inequality in $\Omega_{\delta}.$ 
In this case, since $0<\lambda<1$ and $|Dd|\leq 1,$ (\ref{Bounds}) reads
\begin{eqnarray}
-\Delta \zeta + |D\zeta|^{m}+\lambda \zeta -f
&\geq&-(1-\alpha)Md^{\alpha -2}+Md^{\alpha -1}\Delta d -\frac{M}{\alpha}d^{\alpha}-\|f\|_\infty+
K\nonumber.
\end{eqnarray} As done before,
it suffices to take the constants $M$ as in (\ref{M-bounds}) and $K$ such that
\begin{equation}\label{K-bounds*}
K>M\biggl[(1-\alpha)\|d^{\alpha -2}\|_\infty+\|d^{\alpha -1}\Delta d\|_\infty +\frac{1}{\alpha}\|d^{\alpha}\|_\infty\biggr]+3\|f\|_\infty
\end{equation} to obtain $-\Delta \zeta + |D\zeta|^{m}+\lambda \zeta
\geq f \hbox{ in } \Omega_\delta$ thus ending the proof that $\zeta$ is a supersolution of
$S(\Omega,f,R,\lambda)$ with state constraint on $\partial\Omega.$ A final application of the strong comparison result for $S(\Omega,f,R,\lambda)$ to
$-\frac{\|f\|_{\infty}}{\lambda},$ $u_{R, \lambda}$ and $\zeta$ easily yields the estimate (\ref{R-Bounds}).
\end{proof}
We continue with the
\begin{lemma}\label{121}\text{}
For $0<\lambda<1$ and $m>1,$ there exists a unique viscosity solution $u_{ \lambda}$ 
of the state constraint problem
\begin{equation}\label{121a}
-\Delta u_{ \lambda} + |Du_{ \lambda}|^{m}+\lambda u_{ \lambda}= f
\,\,\,\,\ \text { in }\,\ \Omega
\end{equation}
\begin{equation}\label{121b}
-\Delta u_{ \lambda} + |Du_{ \lambda}|^{m}+ \lambda u_{
\lambda}\geq f \,\,\,\,\ \text { on }\,\ \partial \Omega
\end{equation} and there exists a constant $C_1>0$ such that for all $0<\lambda<1,$ we have
\begin{equation}\label{126}
|\lambda u_{\lambda}|\leq C_1 \text{ in } \overline{\Omega}.
\end{equation}
\end{lemma}
\begin{proof}[\bf Proof of Lemma \ref{121}] 
Let $K$ be such in (\ref{K-bounds*}), we take $R_\lambda:=\frac{K}{\lambda}$ and find from Lemma \ref{130'} that $$u_{R,\lambda}<R\hbox{ on }\overline\Omega\quad \hbox{ for all } R>R_\lambda.$$ By definition of the generalized Dirichlet problem $S(\Omega,f,R,\lambda)$ for $R>R_\lambda,$ we find that  $u_{R,\lambda}$ is a viscosity solution of (\ref{121a})-(\ref{121b}) for all $R>R_\lambda.$ But, from Theorem \ref{13'''}, there exists a unique viscosity solution, $u_\lambda$ of (\ref{121a})-(\ref{121b}), meaning that $u_\lambda=u_{R,\lambda}$ for all $R>R_\lambda.$

Coming back to (\ref{R-Bounds}), it is obvious to see that 
$$-\|f\|_{\infty}\leq\lambda
u_\lambda\leq K \ \hbox{ in } \overline{\Omega}$$ and (\ref{126}) follows by taking $C_1=\max\{\|f\|_{\infty},K\}.$
\end{proof}
Now, we state this important
\begin{proposition}[Local interior gradient bounds on $u_{\lambda}$]\label{122}\text{}\\
Assume $f\in C(\overline{\Omega})$ and $m>1.$ Let $u_\lambda$ be the function defined in Lemma \ref{121}. We have
\begin{equation}\label{130}
\|D u_{\lambda}(x)\|_{\infty}\leq \Lambda
d^{-\frac{1}{m-1}}(x)\ \text{ in }
\Omega
\end{equation} where $\Lambda$ depends on $m,$ $\overline{\Omega}$ and $f.$
\end{proposition}

This result is already proved in \cite{LaLi} and \cite{CapLeoPo}. In \cite{LaLi}, $f$ is required to belong to $W^{1,\infty}(\overline{\Omega})$ and we apply \cite[Theorem IV.1]{LaLi}, by taking $\beta=0$ and $C_1=C_2=\|f\|_\infty,$ 
to obtain Proposition \ref{122}  with $\Lambda$ depending only on $m,$ $\overline{\Omega}$ and $f.$ In \cite{CapLeoPo}, we just need to have $f\in L^\infty(\overline{\Omega})$ and we obtain Proposition \ref{122} from \cite[Theorem 3.1]{CapLeoPo} and $\Lambda,$ in this case, depends on $m$ and $\|d^{\frac{m}{m-1}}_{\partial\Omega}(f-\lambda u_\lambda)^+\|_{L^\infty(\Omega)}.$ But since $\lambda u_\lambda$ is uniformly bounded in $\lambda$ (see (\ref{126})), we obtain a constant $\Lambda$ free of any dependance in $\lambda.$

\bigskip

Now, we would like to let $\lambda$ go to $0.$ Inspired from
\cite{LaLi}, the device used to provide the convergence entirely
relies on the gradient interior bound method coupled with the
Ascoli's result. We choose arbitrary
$x^* \in \Omega_{\delta_0}$ and set
\begin{equation}\label{ergofunction}
v_{\lambda}(x):=u_{\lambda}(x)-u_{\lambda} (x^{*}) \,\,\ \ \ \
\forall x \in  \Omega.
\end{equation}
We start by focussing on the three following useful properties of $v_\lambda$ for any $0<\lambda<1.$
\begin{lemma}[Global bounds on $\Omega_\delta$]\label{deltabound1}\text{}\\
Let $f\in C(\overline{\Omega})$ and $m>2.$ For any $0<\lambda<1$ and for all $0<\delta<\delta_0,$ the function $v_{\lambda}$ is uniformly
bounded on $\Omega_\delta$ and satisfies
\begin{equation}\label{deltabound2}
|v_{\lambda}(x)|\leq \Lambda C_{\Omega}\delta^{-\frac{1}{m-1}}\ \ \ \text{ for all }\  x\in\Omega_{\delta}.
\end{equation}
\end{lemma}
\begin{proof}[\bf Proof of Lemma \ref{deltabound1}] Let $0<\delta<\delta_0,$ for any $x\in \Omega_{\delta},$ since $v_{\lambda}$ is locally Lipschitz by Proposition \ref{122} and (\ref{admissible}), for any $\gamma_{x,x^*}\in\mathcal{A}_{x,x^*}(\Omega_\delta),$  one gets 
\begin{eqnarray}
|v_{\lambda}(x)-v_{\lambda}(x^{*})|&\leq& \int_{0}^{1} \biggl|\frac{d}{ds}v_{\lambda}(\gamma_{x,x^*} (s))\biggr|ds\nonumber \\
&=& \int_{0}^{1}|Dv_{\lambda}(\gamma_{x,x^*} (s))||\dot{\gamma}_{x,x^*}(s)|ds\nonumber\\
&\leq& \Lambda\int_{0}^{1} d^{-\frac{1}{m-1}}(\gamma_{x,x^*} (s))|\dot{\gamma}_{x,x^*}(s)|ds\nonumber\\
&\leq& \Lambda \delta^{-\frac{1}{m-1}}\int_0^1 |\dot{\gamma}_{x,x^*}(s)|ds\nonumber
\end{eqnarray} From (\ref{lenght}) and $v_{\lambda}(x^{*})=0,$ it follows that 
\begin{equation}\label{115}
|v_{\lambda}(x)-v_{\lambda}(x^{*})|\leq  \Lambda \delta^{-\frac{1}{m-1}}\tilde{d}(x,x^*)\leq\Lambda C_{\Omega}\delta^{-\frac{1}{m-1}}.
\end{equation} 
\end{proof}
\begin{lemma}[Lipschitz continuity in $\Omega_\delta$]\label{locallipschitz}\text{}\\
Let $f\in C(\overline{\Omega})$ and $m>2.$ There exists a constant $K_1>0$ depending on $\Omega,$ $m,$ $\delta_0$ and $\Lambda$ (see (\ref{130})) such that for any $0<\lambda<1$ and for all $0<\delta<\delta_0,$
\begin{equation}\label{lipschit}
|v_{\lambda}(x)-v_{\lambda}(y)|\leq K_1\delta^{-\frac{1}{m-1}}|x-y| \ \
\ \ \forall x,y\in \Omega_\delta.
\end{equation}
\end{lemma}
Actually, the local Lipschitz continuity in $\Omega$ or the global one in $\Omega_\delta$ is the best one could have here. Indeed, the Lipschitz continuity up to the boundary is hopeless since the gradients blows up at the boundary (see \cite[Proposition 3.3]{BaDa1} or the discussion in the proof of Lemma \ref{130'}). 
\begin{proof}[\bf Proof of Lemma \ref{locallipschitz}]
Let $x,y\in\Omega_\delta.$ If $|x-y|\geq\frac{{\bar\delta}_0}{2}$ then from (\ref{deltabound2}), we have 
\begin{equation}\label{115**}
\underset{\Omega_\delta\times \Omega_\delta}\sup \frac{|v_{\lambda}(x)-v_{\lambda}(y)|}{|x-y|}\leq
\underset{\Omega_\delta\times
\Omega_\delta}\sup(|v_{\lambda}(x)|+|v_{\lambda}(y)|) \delta_{0}^{-1}\leq
4\Lambda C_{\Omega}^{-1}\delta^{-\frac{1}{m-1}}\nonumber
\end{equation} yielding, for all $x,y\in\Omega_\delta$ with $|x-y|\geq\frac{\delta_0}{2}.$
\begin{equation}\label{115***}
|v_{\lambda}(x)-v_{\lambda}(y)|\leq 4\Lambda C_{\Omega}^{-1}\delta^{-\frac{1}{m-1}} |x-y|.
\end{equation}
If $|x-y|<\frac{\delta_0}{2}$ then for any $\gamma_{x,y}\in \mathcal{A}_{x,y}(\Omega_\delta)$ , one gets:
\begin{eqnarray}
|v_{\lambda}(x)-v_{\lambda}(y)|&\leq& \int_{0}^{1} \biggl|\frac{d}{ds}v_{\lambda}(\gamma_{x,y} (s))\biggr|ds\nonumber\\
&\leq& \int_{0}^{1} |Dv_{\lambda}(\gamma_{x,y} (s))||\dot{\gamma}_{x,y} (s)|  ds\nonumber\\
&\leq& \Lambda \delta^{-\frac{1}{m-1}}\int_{0}^{1} |\dot{\gamma}_{x,y} (s)|  ds\nonumber.
\end{eqnarray} With (\ref{lenght2}), one finally gets 
\begin{equation}\label{boundarydistancefunction}
|v_{\lambda}(x)-v_{\lambda}(y)|\leq \Lambda \delta^{-\frac{1}{m-1}}\tilde{d}(x,y)\leq \Lambda {\bar K}\delta^{-\frac{1}{m-1}}|x-y|
\end{equation} Combining (\ref{115**}) and  (\ref{boundarydistancefunction}), we obtain, the global Lipschitz estimates (\ref{lipschit}) on $\Omega_\delta$ for all $0<\delta<\delta_0$ with $K_1=\Lambda(4 C_{\Omega}/\delta_0+  {\bar K}).$ 
\end{proof}
\begin{lemma}[Global Holder estimates]\label{131}\text{}\\
Let $f\in C(\overline{\Omega})$ and $m>2.$ There exists a constant $K_2>0$ depending on $\Omega,$ $m,$ $\delta_0,$ $C^*,$
$K_{\delta_0}$ and $\Lambda$ such that for any $0<\lambda<1,$ the function $v_{\lambda}$ satisfies the global H\"older estimate
\begin{equation}\label{regularity}
|v_{\lambda}(x)-v_{\lambda}(y)|\leq K_2|x-y|^{\frac{m-2}{m-1}} \ \
\ \ \forall x,y\in\overline{\Omega}.
\end{equation} The constants $C_{\Omega},$ $K_{\delta_0}$ and $\Lambda$ are introduced in  (\ref{lenght}), \ref{lenght2})  and  (\ref{130}) respectively.
\end{lemma} 
\begin{proof}[\bf Proof of Lemma \ref{131}] If $x\in \Omega^{\delta_0},$ there exists a unique point $p_{\partial\Omega}(x) \in {\partial\Omega}$ such that $|x- p_{\partial\Omega}(x) | = d(x)$ and we define $x_t$ for $0\leq t \leq \delta_0$ by
\begin{equation}
x_t :=p_{\partial\Omega}(x)-t n(p_{\partial\Omega}(x))\nonumber
\end{equation} where we recall that $n(p_{\partial\Omega}(x))$ the unit outward normal vector  to $\partial\Omega$ at $p_{\partial\Omega}(x).$ Since $\Omega$ has a $C^2$- boundary, we have (see \cite{Clarke}) 
\begin{equation}\label{pushinside}
d(x_t)=t \quad \hbox{ for all } \ t\in [0,\delta_0].
\end{equation}
In general, this operation will be used in order to ``push" the point $x$ into $\Omega_\delta,$ for some suitable $\delta \geq d(x).$

Let $x,y \in \overline{\Omega},$ only one of the following cases can occur: either $x \in \Omega_{\delta_0}$ and $y \in \Omega_{\delta_0}$ or   $x \in \Omega^{\delta_0}$ and $y \in \Omega_{\delta_0}$ or $x \in \Omega_{\delta_0}$ and $y \in \Omega^{\delta_0}$ or $x \in  \Omega^{\delta_0}$ and $y \in  \Omega^{\delta_0}.$ The most relevant case, being $x \in  \Omega^{\delta_0}$ or $y \in  \Omega^{\delta_0},$ we will only treat it. %
For any $\delta\in [0,\delta_0],$ we have  
\begin{equation}\label{decomposition}
|v_{\lambda}(x)-v_{\lambda}(y)| \leq |v_{\lambda}(x)-v_{\lambda}(x_\delta)| + |v_{\lambda}(x_\delta)-v_{\lambda}(y_\delta)|+|v_{\lambda}(y_\delta)-v_{\lambda}(y)|.
\end{equation}  
By Proposition \ref{122} and (\ref{pushinside}), we have 
\begin{eqnarray}\label{push1}
|v_{\lambda}(x)-v_{\lambda}(x_\delta)| &\leq&\int
_{0}^{\delta}
\biggl |\frac{d}{dt}v_{\delta}(x_t)\biggr |dt \\
&=& \int _{0}^{\delta}|Dv_{\delta}(x_t)|dt \nonumber \\
&\leq&\delta \int _{0}^{\delta}d^{-\frac{1}{m-1}}(x_t)dt \nonumber \\
&=&\Lambda\int _{0}^{\gamma}t^{-\frac{1}{m-1}}dt
=\frac{m-1}{m-2}\Lambda\delta^{\frac{m-2}{m-1}}\nonumber.
\end{eqnarray} Knowing that $x_\delta, y_\delta \in {\overline \Omega_{\delta}},$ it follows from (\ref{lipschit}) that
\begin{equation}\label{insideholder}
|v_{\lambda}(x_\delta)-v_{\lambda}(y_\delta)| \leq K_1\delta^{-\frac{1}{m-1}}|x-y|.
\end{equation} Plugging  (\ref{push1}) and  (\ref{insideholder}) into (\ref{decomposition}), one gets:
\begin{equation}\label{holder}
|v_{\lambda}(x)-v_{\lambda}(y)|\leq 2\Lambda_m\delta^{\frac{m-2}{m-1}}+ K_1|x-y|\delta^{-\frac{1}{m-1}}:=g(\delta).
\end{equation}
After some computations, we find that the function $g$ achieves a minimum at $${\tilde \delta}=\frac{K_1|x-y|}{2(m-1)\Lambda}\leq\frac{K_1\hbox{diam}(\Omega)}{2(m-1)\Lambda}.$$ Up to choose the constant $\Lambda$ large enough,
we can assume that ${\tilde \delta}\leq \delta_0.$ Therefore, coming back to (\ref{holder}), one gets
\begin{equation}
|v_{\lambda}(x)-v_{\lambda}(y)| \leq g({\tilde \delta})=K_2|x-y|^{\frac{m-2}{m-1}} \nonumber
\end{equation} where $$ K_2:=\frac{K_1^{\frac{m-2}{m-1}}}{(m-2)(2(m-1)\Lambda)^{-\frac{1}{m-1}}}+K_1.$$
\end{proof}
Now, we have all the tools necessary to let $\lambda$ go to $0$ and then obtain the existence and uniqueness of the pair $(c,u_\infty).$
\begin{lemma}[Uniform limit of $v_{\lambda}$]\label{127}\text{}\\
There exists $(c,u_{\infty}) \in \mathbb{R}\times
C^{0,\frac{m-2}{m-1}}(\overline{\Omega})\cap
W_{loc}^{1,\infty}(\Omega),$ a pair solution of $Erg(\Omega,f,c).$ Moreover, if $(v,c)$ and $(\tilde{v},\tilde{c})$ are two pairs of
solutions of $Erg(\Omega,f,c)$ such that $v$ and $\tilde{v}$ both belongs to $C^{0,\frac{m-2}{m-1}}(\overline{\Omega})\cap
W_{loc}^{1,\infty}(\Omega),$ then $c= \tilde{c}$ and  $
\tilde{v}=v+C \ \text{ for some } C\in \mathbb{R}.$
\end{lemma}
\begin{proof}[\bf Proof of Lemma \ref{127}] We divide this proof into several parts.

\noindent 1. Existence of the pair $(c,u_\infty).$

\noindent For all $x\in \overline{\Omega},$ plugging $y=x^*$ in (\ref{regularity}) and knowing that $v_\lambda(x^*)=0$  we have:
\begin{eqnarray}\label{globalbounds}
|v_\lambda(x)|\leq K_2|x-x^*|^{\frac{m-2}{m-1}}\leq K_2(diam(\overline{\Omega}))^{\frac{m-2}{m-1}}.
\end{eqnarray} Using (\ref{regularity}) and (\ref{globalbounds}), we find that the family
$\{v_{\lambda}\}_{\lambda>0}\subset C^{0,\frac{m-2}{m-1}}(\overline{\Omega})\cap
W_{loc}^{1,\infty}(\Omega)$ is uniformly globally bounded and equicontinuous. From (\ref{126}), we note that the sequence
$\{\lambda u_{\lambda}(x^*)\}_{\lambda}\subset \mathbb{R}$ is uniformly bounded. From the Azerla-Ascoli compactness
criterion and the compactness of $\overline{\Omega},$ we can choose a sequence $\{\lambda_j\}_{j\in \mathbb{N}}\subset(0,1)$ such that as $j\to+\infty,$
\begin{eqnarray}\label{}
&&\lambda_j\to 0 , \ \ \ -\lambda_j u_{\lambda_j}(x^*)\to c\nonumber \\
&&v_{\lambda_j}(x)\to u_{\infty}(x)\ \ \text{ uniformly for all } x\in\overline{\Omega} \nonumber
\end{eqnarray} for some real constant $c$ and some $u_{\infty} \in
C^{0,\frac{m-2}{m-1}}(\overline{\Omega})\cap
W_{loc}^{1,\infty}(\Omega)$ such that $u_{\infty}(x^*)=0.$ Using the 
stability result for viscosity solutions (see \cite{Barles1}, \cite{BaPe} or \cite{CranILi}), we find that
$u_{\infty}$ is a viscosity solution of $Erg(\Omega,f,c).$

Hereafter, we denote by $PSC(\mathcal{O},\phi,\psi)$ the following parabolic problem with state constraints 
\begin{equation}\label{6} \,\,\,\,\,\,\,\, \left\{
 \begin{array}{rl}
w_{t}-\Delta w + |Dw|^{m} &=\phi \quad\hbox{ in }\quad \mathcal{O} \times [0,T]\\
w&=\psi \quad\hbox{ in }\quad\overline{\mathcal{O}} \times \{0\}\\
w_{t}-\Delta w+ |Du|^{m} &\geq \phi \quad\hbox { on }\quad
\partial\mathcal{O} \times [0,T],
\end{array} 
\right.
\end{equation} for any $T>0.$ We recall that Theorem \ref{13} holds for any generalized initial boundary-value problem of  $E(\Omega,\phi,\chi,\psi,\lambda)$- type, it therefore applies to $PSC(\mathcal{O},\phi,\psi).$ 

\noindent 2. Uniqueness of the ergodic constant $c.$

Noticing that $v+ct$ and
$\tilde{v}+\tilde{c}t$ are solutions solutions of $PSC(\Omega,f+c,v)$ and $PSC(\Omega,f+c,\tilde{v})$ respectively,
we find from (\ref{31}) that
\begin{eqnarray}\label{131'}
(c-\tilde{c})t +(v-\tilde{v})^+(x) \leq \|(v-\tilde{v})^+\|_{\infty}\nonumber
\end{eqnarray} for all $(x,t) \in \overline{\Omega} \times [0,T].$ In particular for $\bar{x}\in \overline{\Omega}$ such that $$(v-\tilde{v})^+(\bar{x})=\max_{\overline{\Omega}}{(v-\tilde{v})^+}.$$ It therefore follows that 
$(\tilde{c}-c)t\leq 0$ meaning that $\tilde{c}\leq c.$ By inverting the roles of $\tilde{c}$ and $c,$ one gets
$\tilde{c}=c.$ 

\noindent 3. Uniqueness of $v$ up to a constant.

\noindent Translating, if necessary $ \tilde{v}$ to $\tilde{v}+K$ for any $K> \|v\|_{\infty}+\|\tilde{v}\|_{\infty},$ we will assume in the rest of the proof that $\tilde{v}\geq v$ on $\overline{\Omega}.$ For all $0<\delta<\delta_0,$ we first argue in $\Omega^\delta$ and then in $\Omega_\delta,$ the aim being to prove that $\tilde{v}-v$ achieves its global maximum on $\overline{\Omega}$ inside $\Omega$ and to conclude by means of the Strong Maximum Principle (Lemma \ref{linear*}) that $\tilde{v}-v$ is constant in $\overline{\Omega}.$ 

Let $0<\delta<\delta_0,$ we first consider $\Omega^\delta.$ 
Let $0<\nu^{m-1}<\frac{1}{m-1}$ and $\alpha:=\frac{m-2}{m-1},$ we denote by $\chi$ the function defined on $\Omega^{\delta}$ by
\begin{eqnarray}
\chi(x):=\frac{\nu}{\alpha} d^{\alpha}(x)-C
\end{eqnarray} where $C$ is a nonnegative constant to be choosen such that $$\chi<M_1:=\|v\|_\infty+ \|\tilde{v}\|_\infty+1 \quad\hbox{ on }\partial\Omega.$$ We claim that $\chi$ is a (smooth) strict subsolution of $\mathcal{S}(\Omega^{\delta},f+c,M_1,0).$ Indeed,
\begin{eqnarray}\label{strictsub}
-\Delta \chi + |D\chi|^{m}-f-c&=[\nu^{m-1}
-\frac{1}{m-1}]\nu d^{-\frac{m}{m-1}}-\nu d^{-\frac{1}{m-1}}\Delta d-f-c.\nonumber
\end{eqnarray} Up to choose $\delta$ small enough, we see that
the right-hand side of the above equality is strictly negative since $d^{-\frac{m}{m-1}}$ is the leading term and $[\nu^{m-1}
-\frac{1}{m-1}] < 0$. Thus, $\mathcal{S}(\Omega^{\delta},f+c,M_1,0)$ has a strict subsolution and we can therefore apply Corollary \ref{scr2} with $\mathcal{O}:={\Omega^{\delta}}.$ Since $v$ and $\tilde{v}$ are two viscosity solutions of $Erg(\Omega,f, c),$ it follows that $v$ and $\tilde{v}$ both satisfies $S(\Omega,f+c,M_1,0)$ and recalling that $\tilde{v}\geq v$ on $\overline{\Omega},$  (\ref{compar1}) yields
$$\underset{{\Omega^{\delta}}}\sup (\tilde{v}-v)=\|(\tilde{v}-v)^+\|_\infty\leq \|(\tilde{v}-v)_{|_{{\partial\Omega^\delta}}}^+\|_{\infty}=\underset{{\partial\Omega^\delta}}\sup (\tilde{v}-v),$$
but given that $\partial\Omega^{\delta}=\partial\Omega \cup \Gamma_\delta$ and that, on $\partial\Omega$, the state-constraint boundary condition  for $v$ and $\tilde{v}$ allows to assume that we have the same Dirichlet boundary condition, namely a constant $M_1\gg 1,$ we  find that
\begin{equation}\label{maxprinc4}
 \underset{\overline{{\Omega^{\delta}}}}\sup (\tilde{v}-v)=\underset{\Gamma_\delta}\sup (\tilde{v}-v).
\end{equation} where $\Gamma_{\delta}:=\{x\in \overline{\Omega} : d(x)= \delta\}.$

Next, we consider $\Omega_\delta.$	We first claim that $v \in W^{2,p}(\Omega_\delta)$ for any $p$ and for any $\delta$. Indeed, since $v\in W^{1,\infty}(\Omega_\delta)$ (see (\ref{130})), by putting $M_2:=\Lambda \delta^{-\frac{m}{m-1}},$ 
we use classical arguments in \cite{GilTru} to show that the unique solution $v$ of
\begin{equation}\label{schauder1}
\left\{
\begin{array}{rl}
-\Delta v+|Dv|^m \wedge M_2+v&=f+c+v\quad\hbox{ in } \Omega_\delta\\
v&=v \quad\quad\quad\hbox{ on } \partial\Omega_\delta=\Gamma_\delta
\end{array}
\right.
\end{equation} is also the unique weak solution of
\begin{equation}\label{schauder2}
\left\{
\begin{array}{rl}
-\Delta w+|Dw|^m \wedge M_2 +w&=G\quad\hbox{ in } \Omega_\delta\\
w&=v \quad\hbox{ on } \partial\Omega_\delta.
\end{array}
\right.
\end{equation} where $G(x)=f(x)+c+v(x)$ in $\Omega$. Indeed, knowing that $G$ and the gradient term are in $L^\infty(\Omega_\delta)$ and therefore in $L^p(\Omega_\delta)$ for all $p$, we have Calderon-Zygmund estimates for Equation~(\ref{schauder2}) and (regularizing the boundary data if necessary) we can use the Leray-Schauder fixed point theorem to solve (\ref{schauder2}) (cf.  \cite{GilTru}). We obtain a solution $w$ which is in $W_{loc}^{2,p}$ for any $p$, and continuous up to $\partial\Omega_\delta$. Moreover, because of its regularity, $w$ is also a viscosity solution of (\ref{schauder1}) and, by uniqueness, $w=v$ in $\Omega_\delta$.

Since the above argument can be used for any $\delta$, we deduce that $v\in W^{2,p}(\Omega_\delta)\cap C(\overline{\Omega_\delta})$ for any $p$ and $\delta$. Moreover
$$v \in W^{2,p}(\Omega_\delta)\hookrightarrow
C^{1,\nu}(\Omega_\delta)\ \ \text{ for all } p>N \text{ and } \nu=1-\frac{N}{p}.$$ 

Of course, the same regularity result is true for $\tilde{v}$ and $\tilde{v}-v$ is a weak subsolution of a linear equation : indeed,
\begin{eqnarray}\label{linearization2}
0&=&-\Delta (\tilde{v}-v)+ |D\tilde{v}|^{m}- |Dv|^{m}\nonumber \\
&=&-\Delta (\tilde{v}-v) + \int
^{1}_{0}\dfrac{d}{dt} |t D\tilde{v} +(1-t)Dv|^{m}dt
\nonumber \\
&=&-\Delta (\tilde{v}-v) +  \mathcal{B}(x)\cdot D(\tilde{v}-v)
\end{eqnarray} where $\mathcal{B}$ is defined as follow
\begin{equation}\label{linear}
\mathcal{B}(x):=\int ^{1}_{0}m|t D \tilde{v}(x) +(1-t)Dv(x)|^{m-2}[t D\tilde{v}(x)+(1-t)Dv(x)]dt.
\end{equation}
But $\mathcal{B}\in L^\infty(\Omega_\delta)$, and more precisely $|\mathcal{B}(x)|\leq m\delta^{-1},$ and we can apply the classical maximum principle 
(see \cite{GilTru}) to (\ref{linearization2}) in  $\Omega_\delta$ and obtain 
\begin{equation}\label{maxprinc3'}
\underset{\overline{\Omega_{\delta}}} \sup(\tilde{v}-v) =\underset{\Gamma_\delta}\sup (\tilde{v}-v).
\end{equation}

Combining (\ref{maxprinc4}) and (\ref{maxprinc3'}), we finally get that the global maximum of $\tilde{v}-v$ on $\overline{\Omega}$
is achieved at some $x_0$ such that for all $d(x_0)=\delta.$ Arguing as above, we see that $\tilde{v}-v$ is a viscosity subsolution of (\ref{linearization}) with $\mathcal{O}:=\Omega_{\delta/2}$ and $C=2m/\delta.$ Applying Lemma \ref{linear*} in $\Omega_{\delta/2},$ we find that $\tilde{v}-v$ is constant in $\Omega_{\delta/2}$ for all $0<\delta<\delta_0.$ But since $\tilde{v}, v\in C^{0,\frac{m-2}{m-1}}(\overline{\Omega}),$ it follows that $\tilde{v}-v$ is constant in $\overline{\Omega}$ by continuity.
\end{proof}
We continue with

\begin{proposition}[Characterization of the constant $c$]\label{ergodic3}\text{}\\
The ergodic constant $c$ introduced in Theorem \ref{112} is
characterized as follows:
\begin{equation}\label{ergodic4}
c=\inf\{a\in \mathbb{R} : \hbox{ there exists } u_a\in C(\overline{\Omega}) \hbox{ with } -\Delta u_a +|Du_a|^{m}\leq f+a  \hbox{  in }\Omega
\}.
\end{equation}
\end{proposition}
\begin{proof}[\bf Proof of Proposition \ref{ergodic3}] We write $\lambda^*$
for the right-hand side of (\ref{ergodic4}) and remark that
$\lambda^*\leq c.$ To see this, we just note that $c\in
\mathbb{R}$ and $u_\infty\in C(\overline{\Omega})$ is a viscosity
subsolution of $-\Delta u_\infty +|Du_\infty|^{m}\leq f+c$ in
$\Omega.$ Assume that $\lambda^*<c,$ there therefore exists a
positive constant $a\in\ ]\lambda^*,c[$ and a continuous subviscosity
solution $u_a$ of $-\Delta u_a +|Du_a|^{m}\leq f+a\ $ in $\Omega.$
We notice that  for any $t\geq 0,$ the function
$u_a(x)+(c-a)t$ is a subsolution of the state constraint problem
$PSC(\Omega,f+c,u_a)$ where as $u_\infty$ is a supersolution of $PSC(\Omega,f+c,u_\infty)$ and
by the comparison result in Theorem \ref{13}, we
find that 
\begin{equation}\label{characterize}
u_a(x)+(c-a)t\leq u_\infty \quad\hbox{ in }\ \Omega\times[0,T]  \quad\hbox{ for any } T>0.
\end{equation} But since $u_\infty$ and $u_a$ are bounded in $\overline{\Omega},$ by choosing $T>0$ large enough, we see that (\ref{characterize})
clearly leads to a contradiction.  Thus, we have $c=\lambda^*,$ ending this proof.
\end{proof}
We end with this
\begin{remark}\label{link}\text{}\rm
\item[\ (i)] We necessarily have $c<0$ when $f>0$ on ${\overline\Omega}.$ Indeed, $u_\infty$ being continuous on $\overline{\Omega}$, it achieves a minimum at some point $x\in \overline{\Omega};$ applying Definition \ref{visc} (ii) with $\varphi=0$, we obtain $0\geq f(x)+c$ which yields $c<0$ since $f(x)>0.$ Notice that we have used the fact that $u_\infty$ is a supersolution of the equation up to the boundary.
\item[\ (ii)] If $c>0,$ then the generalized Dirichlet problem $S(\Omega,f,g,0)$ has no bounded viscosity solution. Indeed, let us assume on the contrary that there exists a bounded continuous viscosity solution, say $\psi,$ of $S(\Omega,f,g,0).$  It is obvious that $\psi(x)+ct$  is subsolution of $E(\Omega,f+c,g+ct,\psi,0)$ and $u_\infty$ is a supersolution of $E(\Omega,f+c,g+ct,u_\infty,0)$ since $u_\infty$ solves the state constraint problem $Erg(\Omega,f,c).$ By Corollary \ref{scr1}, we would have,  for any $T>0,$ $$\|(\psi+ct-u_\infty)^+\|_\infty\leq \|\psi-u_\infty)^+\|_\infty \quad\hbox{ for all } (x,t)\in \Omega\times[0,T].$$  Since $\psi$ and $u_\infty$ are both bounded in $\overline{\Omega},$ by choosing $T>0$ large enough, we reach a contradiction by taking $t=T.$
\end{remark}
\section{Convergence as $t\to +\infty$ to the stationary or ergodic problem }\label{section4}
The goal of this section is to describe the asymptotic behavior of the solution $u$ of the generalized initial boundary-value problem
$E(\Omega,f,g,u_0,\lambda)$ in connection with the generalized boundary-value problem $S(\Omega,f,g,\lambda)$ and the ergodic problem $Erg(\Omega,f,c).$  To do so, we will prove the following:
\begin{theorem}[Convergence result]\label{112'}\text{}\\
Let $m>2,$ $f\in C(\overline{\Omega}),$ $u_0\in C(\overline{\Omega})$ and $g\in C(\partial\Omega).$ Let $u$ be the unique continuous viscosity solution of $E(\Omega,f,g,u_0,\lambda).$ 
\item[{\ \bf (A)} -] Let $\lambda>0.$ Then $u$ is uniformly bounded on
$\overline{\Omega}\times [0, +\infty)$ and
\begin{equation}\label{ergodic***}
u(x,t)\to u_1(x)\quad\hbox{ uniformly on } \overline{\Omega} \quad\hbox{ as } t \to +\infty
\end{equation} where $u_1$ is the unique bounded viscosity solution $S(\Omega,f,g,\lambda).$
\item[{\ \bf (B)} -] Assume $\lambda=0.$ Let $c \in \mathbb{R}$ and  $u_\infty\in  C^{0,\frac{m-2}{m-1}}(\overline{\Omega}) \cap W_{loc}^{1,\infty}(\Omega)$ be such that $u_\infty$ is a viscosity solution of the ergodic problem $Erg(\Omega,f,c).$ Then, 
\item[\ (i)] If $c<0,$ then $u$ is uniformly bounded on
$\overline{\Omega}\times [0, +\infty)$ and
\begin{equation}\label{ergodic*}
u(x,t)\to u_2(x)\quad\hbox{ uniformly on } \overline{\Omega} \quad\hbox{ as } t \to +\infty
\end{equation} where $u_2$ is the unique bounded viscosity solution $S(\Omega,f,g,0).$
\item[\ (ii)] If $c>0$ then $S(\Omega,f,g,0)$ has no viscosity
solution. Moreover, the function $u(x,t)+ct$ is uniformly bounded
on $\overline{\Omega}\times [0, +\infty)$ and
\begin{equation}\label{ergodic}
u(x,t)+ct\to u_{\infty}(x) + K_1\quad\hbox{ uniformly on } \overline{\Omega}\quad\hbox{ as } t \to +\infty
\end{equation} for some constant $K_1$ depending on $f,$ $c,$ $u_0,$ and $g.$
\item[\ (iii)] If $c=0,$ then any solution of $S(\Omega,f,g,0)$ has the form $u_\infty-C$ for some constant $C\in \R$ such that $u_\infty-C\leq g$ on $\partial\Omega$. Moreover 
\begin{equation}\label{ergodic**}
u(x,t)\to u_\infty(x) - \tilde C\quad\hbox{ uniformly on } \overline{\Omega}\quad\hbox{ as } t \to +\infty
\end{equation} for some constant $\tilde C$ depending on $f,$ $c,$ $u_0,$ and $g,$ with $u_\infty-\tilde C\leq g$ on $\partial\Omega$.
\end{theorem}
The rest of this section is devoted to the proof of Theorem \ref{112'}. 
\subsection{The case where $u_0\in C^2(\overline{\Omega})$}\label{subsection1}\text{}

To perform the proof of the convergence result in this case, we introduce the following usefull result
\begin{proposition}[Global H\"older estimates]\label{136**}\text{}\\
Let $f\in L^\infty(\Omega),$ $u_0\in C^2(\overline{\Omega})$ and $g\in C(\partial\Omega).$ 
Then the unique continuous viscosity solution $u$  of $E(\Omega,f,g,u_0,\lambda)$ satisfies
\begin{equation}\label{regularity*}
|u(x,t)-u(y,t)|\leq M|x-y|^{\frac{m-2}{m-1}} \quad\hbox{ for all }
x,y\in\overline{\Omega} \hbox{ and } t\geq0
\end{equation}where $M$ is a positive constant independent of $t.$
\end{proposition}
\begin{proof}[\bf Proof of Proposition \ref{136**}]\text{} It is sufficient to prove that $u(x,t)$ is a Lipschitz continuous 
function with respect to its $t$-variable, that is there exists a constant $C^*>0$ such that
\begin{equation}\label{timebounds}
\|u_t(x,t)\|_\infty\leq C^* \quad\hbox{ for all } (x,t)\in\overline{\Omega}\times [0,+\infty).
\end{equation} Indeed, if (\ref{timebounds}) holds, we remark that $u$ is a viscosity subsolution of $-\Delta u+|Du|^m+\lambda u\leq K$ in $\Omega$ with $K:=\|f\|_\infty+C^*$ and then we just apply \cite[Theorem 2.7]{CapLeoPo} and (\ref{regularity*}) follows.

To prove (\ref{timebounds}), we take a small $h>0$ and by denoting  by $u_\xi(x):=u(x,\xi)$ for all $x \in \Omega,$ we remark $u(\cdot,t+h)$ and $u(\cdot,t)$ are viscosity solution of $E(\Omega,f,g,u_h,\lambda)$ and $E(\Omega,f,g,u_0,\lambda)$ respectively. We find from the Corollary \ref{13*} that
$$\|(u(x,t+h)-u(x,t))^+\|_\infty \leq e^{-\lambda t}\|(u_h(x)-u_0(x))^+\|_\infty\leq \|(u_h(x)-u_0(x))^+\|_\infty.$$ Likewise, one gets
$$\|(u(x,t)-u(x,t+h))^+\|_\infty \leq \|(u_0(x)-u_h(x))^+\|_\infty.$$
It therefore follows that
\begin{eqnarray}\label{comparison}
\|u(x,t+h)-u(x,t)\|_\infty  \leq \|u(x,h)-u(x,0)\|_\infty .
\end{eqnarray}
To continue, we need to estimate $\|u(x,h)-u(x,0)\|_\infty .$ To do so, knowing that $u_0\in C^2(\overline{\Omega}),$ we put $$C_*:=\|\Delta u_0\|_\infty+\|D u_0\|^m_\infty+\lambda\|u_0\|_\infty+\|f\|_\infty+\|g\|_\infty+1$$ and see that $\psi_1(x,t):=u_0(x)-C_*t$ and
$\psi_2(x,t):=u_0(x)+C_*t$ are respectively sub- and supersolution of  $E(\Omega,f,u_0,g,\lambda).$ Indeed, 
$$(\psi_1)_t-\Delta\psi_1+|D\Psi_1|^m+\lambda \psi_1-f=-C_*-\Delta u_0+|Du_0|^m+\lambda u_0-\lambda C_*t-f<0.$$ We do the same with $\psi_2$ and note
from Theorem \ref{13} that 
\begin{equation}
u_0(x)-C_*t\leq u(x,t)\leq u_0(x)+C_*t \ \ \text{ for all } x\in \Omega, t>0.
\end{equation} By taking $t=h,$ it therefore follows that $\|u(x,h)-u(x,0)\|_\infty  \leq C_*h.$ Going back to (\ref{comparison}), we obviously obtain
$\|u(x,t+h)-u(x,t)\|_\infty  \leq C_*h$ and thus (\ref{timebounds}) follows by sending $h\to 0.$
\end{proof}
The rest of this section is organized as follows. In Section
\ref{classical}, we give the proof of Theorem \ref{112'}
{\bf (A)} and {\bf (B)}-(i) while the Sections \ref{generalized} and \ref{stateconstraint} are repsectively devoted to the
proof of Theorem \ref{112'} {\bf (B)}-(ii) and {\bf (B)}-(iii).
\subsubsection{The  case where $\lambda>0$ or $c<0$}\label{classical}\text{}

\noindent First, for $\lambda>0,$ we find from Theorem \ref{13'''} that $S(\Omega,f,g,\lambda)$ has a unique viscosity solution which we denote $u_1.$ Moreover, if $c<0,$ then $u_\infty$ is a strict subsolution of (\ref{112a}). Given that $u_\infty$ and $g$ are bounded, for $C>\|u_\infty\|_\infty+\|g\|_\infty,$ we have $u_\infty-C\leq g$ on $\partial\Omega.$ Therefore, $u_\infty-C$ is a strict subsolution of $S(\Omega,f,g,0)$ and by application of Theorem \ref{13'''}, we find that $S(\Omega,f,g,0)$ has a unique viscosity solution which we denote $u_2.$

\noindent Next, we claim that the solution $u$ of $E(\Omega,f,u_0,g,\lambda)$ is uniformly bounded in $\overline{\Omega}\times [0,+\infty).$ To see that, one just remarks that, for all  $(x,t)\in \Omega\times [0,+\infty),$ we have
\begin{equation}\label{301}
u_\infty(x)-C\leq u(x,t)\leq |x-x_{0}|^{2}+\|u_{0}\|_{\infty}+ \|g\|_{\infty}+1 \quad\hbox{ when } \lambda=0
\end{equation} and 
\begin{equation}\label{301'}
-\frac{\|f\|_{\infty}}{\lambda}\leq u(x,t)\leq |x-x_{0}|^{2}+\|u_{0}\|_{\infty}+ \|g\|_{\infty}+1 \quad\hbox{ when } \lambda>0
\end{equation} where  $x_0\in \mathbb{R}^N$ and $B(x_{0},K)\cap \overline{\Omega}=\emptyset$ with
$K>(\|f\|_{\infty}+2N)^{1/m}$ for some $C\in \mathbb{R}^+.$

Finally, we claim that (\ref{ergodic**}) holds. Indeed, let $(u^\varepsilon)_{0<\varepsilon<1}$ be the sequence defined by
\begin{equation}\label{noisy}
u_{\varepsilon}(x,t)=u(x,t/\varepsilon)\ \text{ for all } (x,t)\in
\Omega\times[0,+\infty).
\end{equation} 
Since $(u_{\varepsilon})_{0<
\varepsilon< 1}$ is uniformly boundedness (see (\ref{301'})), we apply the half-relaxed limits method (see \cite{Barles1}) and find that the functions
$$\overline{u}(x)=\underset{\underset{\varepsilon\downarrow 0}{y\to x}}\limsup\
u_{\varepsilon}(y,t)\ \text{ and }\ \
\underline{u}(x)=\underset{\underset{\varepsilon\downarrow  0}{y\to x}}{\liminf}\
u_{\varepsilon}(y,t)$$ are respectively subsolution and supersolution of
$S(\Omega,f,g,\lambda).$ By their very definition, we have $\underline{u}\leq\overline{u}$ in
$\Omega$  but the comparison result for $S(\Omega,f,g,\lambda)$ yields $\overline{u}\leq\underline{u}$ in
$\Omega.$ Moreover, from (\ref{regularity*}), we have $\overline{u}\in C^{0,\frac{m-2}{m-1}}(\overline{\Omega})$ and  deduce that $\overline{u}$  may be continuously extended on  $\overline{\Omega}$ and thus $\overline{u}\leq\underline{u}$ on $\overline{
\Omega}$ by using Theorem \ref{scr1}. Hence $u_1:=\underline{u}=\overline{u}$ on $\overline{\Omega}$ meaning that $(u_{\varepsilon})_{0<\varepsilon< 1}$ converges  to $u_1$ uniformly in $\overline{\Omega}$ as $\varepsilon\downarrow 0$ and (\ref{ergodic**}) follows. We use exactly the same arguments to prove that (\ref{ergodic*}) holds too.

\subsubsection{The $c>0$ case}\label{stateconstraint}\text{}

We start with
\begin{lemma}\label{behavior1}
$u(x,t)+ct$ is uniformly bounded on $\overline{\Omega}\times[0,T]$ for all $T>0.$
\end{lemma}
\begin{proof}[\bf Proof of Lemma \ref{behavior1}] Since $c>0$ and $u_{\infty}$ solves $Erg(\Omega,f,c),$  we find that $u_{\infty}-K$   is a subsolution of $E(\Omega,f+c,g+ct,u_0,0)$ whereas $u_{\infty}+\tilde{K}$ is supersolution of $E(\Omega,f+c,g+ct,u_0,0)$ with state constraint condition on $\partial\Omega$ with $$K>\|u_\infty\|_\infty+\|u_0\|_\infty+\|g\|_\infty\hbox{ and }\tilde{K}>\|u_\infty\|_\infty+\|u_0\|_\infty+\|g\|_\infty.$$ 
Since $u+ct$ is the unique solution of $E(\Omega,f+c,g+ct,u_0,0),$ we find from Theorem \ref{13} that for all $T>0,$
\begin{equation}\label{138'}
u_{\infty}(x)-K \leq u(x,t)+ct\leq u_\infty(x)+\tilde{K}\quad\hbox{ for all } (x,t) \in \Omega\times[0,T]
\end{equation} But $u_\infty, u(\cdot,t)\in C^{0,{\frac{m-2}{m-1}}}(\overline{\Omega})$ for any $t>0,$ hence, by continuous extention up to the boundary, it follows that (\ref{138'}) still holds for all $x\in\overline{\Omega}.$ 
\end{proof}
\begin{remark}\rm It is worth noticing that ``$c>0$"  and ``$u(x,t)+ct$ is
uniformly bounded on $\overline{\Omega}$" are two consistent facts. Indeed, as a consequence of Remark \ref{link}, the ergodic constant $c$ is  necessarily nonnegative when $\min_{\overline{\Omega}} f<0.$ Therefore, coming back to (\ref{valfunct}), for some $f$ such that $\min_{\overline{\Omega}} f<0$ and $|\min_{\overline{\Omega}} f|\gg1,$ it could happen that $u(x,t)\to -\infty.$  In this case, there exists no bounded solution for $S(\Omega,f,g,0)$ as earlier noted in Remark \ref{link}.  To turn around the probably unboundedness of $u,$ we counterbalance by adding an appropriate nonnegative constant $c>0$ in such a way to obtain a reasonable $f+c$ and a function $u+ct$ bounded from below.  
\end{remark}
\begin{lemma}\label{behavior2} (\ref{ergodic}) holds.
\end{lemma}
\begin{proof}[\bf Proof of Lemma \ref{behavior2}] \text{} Hereafter for all
$(x,t)\in \overline{\Omega}\times[0,+\infty),$ we put $$v(x,t):=u(x,t)+ct $$ and split the proof into
several parts.

\noindent 1.  The function $v(\cdot,\cdot + t)$ solves $E(\Omega,f+c,g+c(\cdot+t),v(\cdot,t),0)$ whereas $u_{\infty}$ solves $Erg(\Omega,f,c)$ and therefore 
it is a supersolution of $E(\Omega,f+c,g+c(\cdot+t),u_\infty,0)$ We use (\ref{31}) and obtain, for all $x\in\overline{\Omega}$
and $s\geq t\geq 0$
\begin{equation}
\underset{x\in\overline{\Omega}}\max(v(x,s)-u_{\infty}(x))\leq \underset{x\in\overline{\Omega}}\max
(v(x,t)-u_{\infty}(x)).\nonumber
\end{equation} Therefore, the function
\begin{equation}\label{invariant}
m(t)=\underset{x\in\overline{\Omega}}\max
(v(x,t)-u_{\infty}(x))
\end{equation} is nonincreasing and bounded from below, thus we have
\begin{equation}
m(t)\to \overline{m}\quad\hbox{ as } t\to+\infty.
\nonumber
\end{equation} 

From the uniform boundedness of $v(x,t)$, (\ref{regularity*}) and the half-relaxed limits method, we find that
$$\overline{v}(x)=\underset{\underset{t\to +\infty}{y\to x}}\limsup\ v(y,t)=\underset{t\to +\infty}\limsup\ v(x,t)$$ is a subsolution of $Erg(\Omega,f,c).$
Moreover
\begin{equation}\label{choosemaximum1}
\underset{x\in\overline{\Omega}}\max(\overline{v}(x)-u_{\infty}(x))=\overline{m}
\end{equation} 
Indeed, on one hand, we clearly have $\overline{v}(x) - u_{\infty}(x) \leq \overline{m}$ on $\overline{\Omega}$ by just taking the $\limsup$ in $t$ in the inequality $v(x,t)-u_{\infty}(x) \leq m(t)$. On the other hand, if $(t_n)_n$ is any sequence such that $t_n\to+\infty$ and if $x_n$ is chosen such that
$$ v(x_n,t_n) - u_{\infty}(x_n) = m(t_n)\; ,$$
($x_n$ exists since $v(\cdot, t_n)-u_{\infty}(\cdot)$ is continuous on the compact set $\overline{\Omega}$), we can assume without loss of generality that $x_n \to x_\infty$ and, 
by definition of $\overline{v},$
$$ \overline{m}=\limsup_n \,m(t_n) =\limsup_n \,(v(x_n,t_n) - u_{\infty}(x_n)) \leq \overline{v}(x_\infty) - u_{\infty}(x_\infty)\; .$$

Next we would like to argue as in the proof of Lemma \ref{127} (part 3) to obtain, by means of the Strong Maximum Principle, that
\begin{equation}\label{convergence1}
\overline{v}(x)=u_{\infty}(x)+\overline{m}\quad\hbox{ for all } x\in\overline{\Omega}.
\end{equation} 
There is no difficulty to repeat the argument on $\Omega^\delta$, but, for the argument in $\Omega_\delta$, a priori $\overline{v}$ is not regular enough since it is only a continuous viscosity subsolution of $Erg(\Omega,f,c).$
To turn around this difficulty, we use the convexity of $p \mapsto |p|^m$ and the regularity of $u_\infty$, which yields (at least formally)
$$ |D \overline{v}|^m \geq  |D u_\infty |^m + m |D u_\infty |^{m-2}D u_\infty\cdot (D \overline{v}-D u_\infty)\; .$$
With this argument and using the local bounds on $|D u_\infty |$, it is easy to justify that $\overline{v}-u_\infty$ is a subsolution of an equation like (\ref{linearization}) in $\Omega_\delta$ for any $\delta$ and the proof of Lemma \ref{127} extends with this argument, using Lemma~\ref{linear*}.

\begin{remark}
At this stage, the use of one half-relaxed limit $\overline{v}$ just gives us a partial convergence. To catch a uniform convergence, we have to re-argue as above with the other half-relaxed limit $\underline{v}(x)=\underset{\underset{t\to +\infty}{y\to x}}\liminf\ v(y,t).$ By using the minimum and the strong minimum principles, we obtain 
$$\underline{v}(x)=u_{\infty}(x)+\underline{m}\quad\hbox{ for all } x\in\overline{\Omega}$$ for some constant $\underline{m}.$ But since $Erg(\Omega,f,c)$ has many solutions, there is no reason to guess that $\overline{m}=\underline{m}.$ So we cannot conclude in this way and we turn around that difficulty by using a combination of arguments relying on the H\"older estimates for $v$ and the Strong Maximum Principle for parabolic equations.
\end{remark}

\noindent 2. Now we choose any point $\xb \in \Omega$. By definition of $\overline{v},$ there exists a sequence $t_n \to + \infty$ such that $v(\xb,t_n)\to \overline{v}(\xb)$ as $n\to+\infty;$ we can assume without loss of generality that $t_n \geq 1$ for any $n$. 

Next we consider the sequence $(v(\cdot,t_n -1))_{n}$. From Proposition \ref{136**}, we can apply Ascoli's Theorem and extract a subsequence $(v(\cdot,t_{n'}-1))_{n'}$ which converges in $C(\overline{\Omega}),$ namely
\begin{equation}\label{cauchy3}
v(\cdot,t_{n'}-1)\underset{n'\to+\infty}\to v_0  \quad\hbox{ uniformly on } \overline{\Omega}.
\end{equation} Moreover, for large $n',$ we define the functions
\begin{equation}\label{change}
w_{n'}(y,s):=v(y,s+t_{n'}-1) \quad\hbox{ for all } y\in \Omega, s>0.
\end{equation} Knowing that we are dealing with large $n'$ and since $v$ is uniformly bounded (see Lemma \ref{behavior1}), we find that 
$$v(y,s+t_{n'}-1) <  g(y)+c(s+t_{n'}-1) \quad\hbox{ for all } (y,s)\in \partial\Omega\times[0,+\infty).$$
It follows that $w_{n'}$ is a viscosity solution of the parabolic state-constraint problem $PSC(\Omega,f+c,v(x,t_{n'}-1)).$ From (\ref{31}), for any $p',q'>n',$ we obtain
\begin{equation}\label{cauchy1}
\| w_{p'}-w_{q'} \|_\infty\leq
\| v(\cdot,t_{p'}-1)-v(\cdot,t_{q'}-1)\|_\infty\underset{{p'},{q'}\to+\infty}\to 0\nonumber.
\end{equation}
Therefore $(w_{n'})_{n'}$
is a Cauchy sequence in $C(\overline{\Omega}\times[0,+\infty))$ for large $n'$ and
\begin{equation}\label{cauchy2}
w_{n'}(y,s)\underset{n'\to+\infty}\to w(y,s) \quad\hbox{
uniformly for all }\ y\in \Omega,\ s>0.
\end{equation} 
By stability, $w$ is a viscosity solution of $PSC(\Omega,f+c,v_0).$ We use the uniform convergence of the $(w_{n'})_{n'}$ to pass to the limit in 
$$m(s+t_{n'}-1)= \underset{y\in\overline{\Omega}}\max
(w_{n'}(y,s)-u_{\infty}(y))$$ and obtain
\begin{equation}\label{timeconst}
\overline{m}=\underset{x\in\overline{\Omega}}\max
(w(x,t)-u_{\infty}(x)) \hbox{ for any }\ \ t\geq0.
\end{equation} Hence the function $$t\mapsto\underset{x\in\overline{\Omega}}\max
(w(x,t)-u_{\infty}(x))$$ is constant on $\mathbb{R}_+.$ 

Since $v(\xb,t_n)\to \overline{v}(\xb)$ as $n\to+\infty$, $w_{n'}(\xb,1)= v(\bar{x},t_{n'}) \to \overline{v}(\xb)$ as $n' \to +\infty.$ Hence, by (\ref{convergence1})
$$  w(\xb,1)=\underset{n'\rightarrow+\infty}\lim w_{n'}(\xb,1)=\overline{v}(\xb)=u_\infty(\xb)+\overline{m} \; .$$

\noindent 3. Using (\ref{timeconst}), we find that, for any $t\geq0,$
\begin{equation}
\underset{x\in\overline{\Omega}}\max
(w(x,t)-u_{\infty}(x))=\overline{m}= w(\xb,1)-u_\infty(\xb), \nonumber
\end{equation}
meaning that the function $$(x,t)\mapsto w(x,t)-u_{\infty}(x)$$ achieves its global maximum on $\overline{\Omega}\times(0,+\infty)$ at $(\xb,1)$ with $\xb \in \Omega.$ 

In order to conclude, we use the same argument as in Step~1 above : the convexity of $p\mapsto|p|^m$ and (\ref{130}) yield that $w(x,t)-u_{\infty}$ satisfies
(\ref{linearization1}) in $\mathcal{O}:=\Omega_\delta$, for any $0<\delta<\delta_0,$ with $C:=m/\delta$.

Applying Lemma \ref{linear*} and taking into account the continuity of $w$ and $u_\infty$ up to the boundary, we find that $(x,t)\mapsto w(x,t)-u_{\infty}(x)$ is constant on $\Omega_\delta\times[0,1]$ for all $0<\delta<\delta_0,$ and therefore on $\overline{\Omega}\times[0,1]$. In particular
\begin{equation}\label{}
w(x,0)=v_0(x)=u_{\infty}(x)+\overline{m}\quad\hbox{ for all } x\in
\overline{\Omega}\nonumber.
\end{equation}
\ 

\noindent 4. We have seen above that $v(x,t)$ is a viscosity solution of the parabolic state-constraint problem if $t$ is large enough, and so is $u_{\infty}(x)+\overline{m}$. Comparing these two solutions for $t \geq t_{n'}-1$ with $n'$ large enough, we have by Corollary \ref{scr2}
$$
\|v(x,t)-(u_{\infty}(x)+\overline{m})\|_\infty \leq \|v(x,t_{n'}-1)-(u_{\infty}(x)+\overline{m})\|_\infty\; .$$
But $\|v(x,t_{n'}-1)-(u_{\infty}(x)+\overline{m})\|_\infty \to  \|v_0(x)-(u_{\infty}(x)+\overline{m})\|_\infty = 0$ as $n'\to +\infty$ and we conclude that, as $t \to +\infty,$
\begin{equation}\label{convergence2} 
v(x,t)=u(x,t)+ct\to
u_{\infty}(x)+\overline{m} \quad\hbox{ uniformly for all }  x \in
\overline{\Omega}.
\end{equation}
This ends the proof of the uniform convergence (\ref{ergodic}) with $K_1=\overline{m}.$
\end{proof}
\subsubsection{The case where $c=0$}\label{generalized}\text{}

\noindent We notice that the function $u_\infty-C$ is a viscosity solution of $S(\Omega,f+c,g,0)$ for all $C$ such that $u_\infty-C\leq g$ on $\partial\Omega.$ Indeed $u_\infty-C$ is a solution of the equation since it solves $Erg(\Omega,f,c)$ and
\begin{equation}\label{relaxeddirichlet1}
\min\{u_\infty-C-g,-\Delta u_\infty+|Du_\infty|^m-f-c\}\leq 0\quad\hbox{ on } \partial\Omega
\end{equation}
holds because $u_\infty-C-g \leq 0$ on $\partial\Omega$ while  
\begin{equation}\label{relaxeddirichlet2}
\max\{u_\infty-C-g,-\Delta u_\infty+|Du_\infty|^m-f-c\}\geq 0\quad\hbox{ on } \partial\Omega,
\end{equation}
because $u_\infty-C$ satisfies a state-contraint boundary condition.

All the solutions of $S(\Omega,f+c,g,0)$ are of the form $u_\infty-C$ with $C$ satisfying the above constraint : indeed, one can repeat the arguments of the proof of Step~3 of Lemma~\ref{127} with the adaptations we have used to deduce (\ref{convergence1}).

In the same way, the convergence result (\ref{ergodic**}) is obtained exactly as the one of (\ref{ergodic}), using the strict subsolution argument in $\Omega^\delta$ and Lemma~\ref{linear*} in $\Omega_\delta$, so we skip it.

\subsection{The general case $u_0\in C(\overline{\Omega})$}\label{subsection2}\text{}

\noindent 1.	We smooth $u_0$ by considering  a sequence $(u_{0,\varepsilon})_{\varepsilon}$ of $C^2$- functions such that $u_0^\varepsilon\to u_0$
in $C(\overline{\Omega})$ as $\varepsilon \to 0^+.$ Moreover, we define a sequence $(g_\varepsilon)_{\varepsilon}\in C^2(\partial\Omega)$ as follow: 
\begin{equation}
g_\varepsilon=u_{0,\varepsilon} \hbox{ on } \partial\Omega\quad\hbox{ for all } 0<\varepsilon<1.\nonumber
\end{equation} Given that $u_0$ and $g$ satisfy (\ref{2}), it is obvious that $g_\varepsilon\to g$
in $C(\partial\Omega)$ as $\varepsilon \to 0^+.$  We denote by $u_\varepsilon$ the unique solution of $E(\Omega,f,u_{0,\varepsilon},g_\varepsilon,0).$

\noindent 2. Going back throught the proof of Lemma \ref{127}, we notice that the pair $(c,u_\infty)$ associated to the ergodic problem $Erg(\Omega,f,c)$ is obtained independently of $u_0$ and $g.$ Since $u_{0,\varepsilon}\in C^2(\overline{\Omega}),$ it follows from what is done above  that Theorem \ref{112'} holds for $u_\varepsilon$ in such a way that, as $t \to +\infty$ and uniformly for all $x\in \overline{\Omega}$
\begin{equation}\label{id-smooth}
\left\{
\begin{array}{rl}
&u_\varepsilon(x,t)\to u_{\varepsilon,1}(x)\quad\quad\quad\quad\quad\hbox{ when } \lambda>0 \\
&u_\varepsilon(x,t)\to u_{\varepsilon,2}(x)\quad\quad\quad\quad\quad\hbox{ when } c<0 \\
&u_\varepsilon(x,t)+ct\to u_{\infty}(x) + C_\varepsilon \quad\hbox{ when } c>0\\
&u_\varepsilon(x,t)\to u_\infty(x)+ C_\varepsilon\quad\quad\quad\hbox{ when } c=0.
\end{array}
\right.
\end{equation} for all $0<\varepsilon<1$ where $u_{\varepsilon,1}$ is the unique solution of $S(\Omega,f,g_\varepsilon,\lambda)$ for $\lambda>0,$ and $C_\varepsilon,$ given by Theorem \ref{112'} depends on $f,$ $c,$ $u_{0,\varepsilon}$ and $g_\varepsilon.$

\noindent 3. Next we compare $u_\varepsilon$ and $u$. From (\ref{31}), we find that, for any $0<\varepsilon<1,$ we have
\begin{equation}
\|u_\varepsilon-u\|_\infty\leq \|u_{0,\varepsilon}-u_{0}\|_\infty +\|g_\varepsilon-g\|_\infty=o_\varepsilon(1) \; . \nonumber
\end{equation} 
Moreover, for $\lambda = 0$ and $c \geq 0$, $\|(u_\varepsilon+ct)-(u+ct)\|_\infty=\|u_\varepsilon-u\|_\infty= o_\varepsilon(1)$.

\noindent 3. In the case when $\lambda = 0$ and $c\geq 0,$ the uniform boundedness of $u+ct$ in $\overline{\Omega}\times [0,+\infty)$ obviously follows from Lemma \ref{behavior1} since the arguments we have used do not depends on the regularity of $u_0.$ We use (\ref{id-smooth}) and apply (\ref{31}) to obtain
\begin{equation}
\|(u_\varepsilon+ct)-(u_{\varepsilon'}+ct)\|_\infty\leq \|u_{0,\varepsilon}-u_{0,\varepsilon'}\|_\infty +\|g_\varepsilon-g_{\varepsilon'}\|_\infty\nonumber
\end{equation} for any $0<\varepsilon, \varepsilon'<1.$ By sending $t\to+\infty,$ it follows that
\begin{equation}
\|C_\varepsilon-C_{\varepsilon'}\|_\infty\leq \|u_{0,\varepsilon}-u_{0,\varepsilon'}\|_\infty +\|g_\varepsilon-g_{\varepsilon'}\|_\infty = o_\varepsilon(1)\nonumber
\end{equation} meaning that the sequence $(C_\varepsilon)_{0<\varepsilon<1}$ is a Cauchy sequence in $\mathbb{R}$
and thus $C_\varepsilon\to C_0$ as $\varepsilon\downarrow0.$ Moreover
\begin{eqnarray*}
\|(u+ct)-(u_\infty+C_0)\|_\infty&\leq&\|(u+ct)-(u_\varepsilon+ct)\|_\infty+ \|(u_\varepsilon+ct)-(u_\infty+C_\varepsilon)\|_\infty\nonumber\\ &&+\|(u_\infty+C_\varepsilon)-(u_\infty+C_0)\|_\infty\\
& \leq & o_\varepsilon (1) + \|(u_\varepsilon+ct)-(u_\infty+C_\varepsilon)\|_\infty\; ,
\end{eqnarray*}
where $o_\varepsilon (1)$ is independent of $t$. In order to conclude, we first fix $\varepsilon$ and take a $\limsup$ in $t$ which yields
$$ \limsup_{t\to+\infty} \|(u+ct)-(u_\infty+C_0)\|_\infty \leq o_\varepsilon (1) \; ,$$
and then we let $\varepsilon$ tend to $0$.

\noindent 4.	When $c<0,$ we apply again Corollary \ref{31} and get  
\begin{equation}
\|u_\varepsilon-u_{\varepsilon,2}\|_\infty\leq \|u_{0}-u_{0,\varepsilon}\|_\infty +\|g-g_{\varepsilon'}\|_\infty\nonumber
\end{equation}  using the facts that the solution of $S(\Omega,f,g,0)$ depends continuously on $g$ (see (\ref{31})) and that the solution of $E(\Omega,f,u_0,g,0)$ depends continuously on $u_0$ and $g$ (see (\ref{compar1})), we send $\varepsilon\downarrow0$ and obtain 
$$\underset{t\to+\infty}\lim u(\cdot,t)=u_2\quad\hbox{ uniformly on } \overline{\Omega}. $$

\noindent 5. The cases $c=0$ and $\lambda>0$ are respectively similar to the cases $c>0$ and $c<0,$ we therefore refer to the proof given above.

\bigskip

{\bf Acknowledgements.} \small{ First, I am grateful to Professor Guy Barles and
Olivier Ley, both from the University FranÁois Rabelais of Tours (France), for
having brought this problem to my attention. This work was deeply influenced by many usefull discussions with them and
valuable suggestions from them. Next, I also  thank
Professor Marcel Dossa Cossy from the University of Yaound\'{e} I (Cameroon)
for many helpful advices he gave me. Finally, I warmly thank Professor Alessio
Porretta from the University of Roma Tor Vergata (Italy) for his advices and useful remarks to improve this work.}

\bigskip
\begin{flushleft}
{\bf Thierry TABET TCHAMBA}\hfill
{\par Laboratoire de Math\'ematiques et  Physique Th\'eorique - UMR CNRS 6083\hfill\\
F\'ed\'eration Denis Poisson - FR 2964 - CNRS, Universit\'{e} Fran\c{c}ois Rabelais - Tours.\hfill\\
and \hfill\\
D\'{e}partement de Math\'{e}matiques - Universit\'{e} de Yaound\'{e} I, Yaound\'e.\hfill
\par}
\end{flushleft}

\noindent Email adress: \textit {tabet@lmpt.univ-tours.fr}

\end{document}